\documentclass[dvips,preprint, 11pt]{imsart}
\usepackage{amsmath}
\usepackage{amsfonts}
\usepackage{mathrsfs}

\usepackage{ amsmath, amsfonts, amssymb, amsthm, amscd}
\usepackage[T1]{fontenc}
\usepackage[english]{babel}
\usepackage{color}
\usepackage{pgf,pgfarrows,pgfnodes,pgfautomata,pgfheaps}
\usepackage{amsmath,amssymb}
\usepackage[latin1]{inputenc}

\usepackage[latin1]{inputenc}
\usepackage{graphicx}

\fontsize{11pt}{14.5pt}\selectfont

\def\Dy#1{\Frac{\partial #1}{\partial y}}

\def\Dy_1y_1#1{\Frac{\partial^2 #1}{\partial y_1^2}}

\newtheorem{Theorem}{Theorem}[part]
\newtheorem{Definition}{Definition}[part]
\newtheorem{Proposition}{Proposition}[part]
\newtheorem{Assumption}{Assumption}[part]
\newtheorem{Lemma}{Lemma}[part]

\newtheorem{Remark}{Remark}[part]

\makeatletter \@addtoreset{equation}{section}

\@addtoreset{Definition}{section}

\@addtoreset{Theorem}{section}

\@addtoreset{Proposition}{section}

\@addtoreset{Assumption}{section}

\@addtoreset{Corollary}{section}

\@addtoreset{Lemma}{section}

\@addtoreset{Remark}{section}

\@addtoreset{Example}{section}

\makeatother \makeatletter

\def \Int{\displaystyle\int}
\def \Frac{\displaystyle\frac}

\def \Sup{\displaystyle\sup}

\def \be{\begin{eqnarray}}
\def \ee{\end{eqnarray}}
\def \b*{\begin{eqnarray*}}
\def \e*{\end{eqnarray*}}

\def \E{\mathbb{E}}

\def \L{\mathbb{L}}
\def \N{\mathbb{N}}
\def \P{\mathbb{P}}
\def \R{\mathbb{R}}

\def \W{\overleftarrow{W}}
\def \B{\overleftarrow{B}}

\def \[{[\,\!\![}
\def \]{]\,\!\!]}

\def \1{{\bf 1}}

\def \ep{\hbox{ }\hfill$\Box$}

\def\Fc{{\cal F}}

\def\Hc{{\cal H}}

\def\Nc{{\cal N}}

\def\Sc{{\cal S}}

\begin{document}

\begin{frontmatter}

\title{ Systems of Reflected  Quasilinear Stochastic PDEs in a Convex Domain}
\runtitle{Quasilinear SPDEs}

  \author{\fnms{Wissal} \snm{Sabbagh}\corref{}\ead[label=e2]{wissal.sabbagh@univ-evry.fr}}
    \thankstext{t2}{The research of the first author benefited from the support of the «Chair Markets in Transition», under the aegis of «Louis Bachelier Finance and Sustainable  Growth» laboratory, a joint initiative of \'Ecole polytechnique, Université d'\'Evry Val d'Essonne and Fédération Bancaire Française.
 }
\address{ LaMME, University of  Evry - Paris-Saclay, \\
91037, Evry, France.
 \\ \printead{e2}}
 \affiliation{University Evry}

\author{\fnms{Tusheng} \snm{Zhang}\corref{}\ead[label=e3]{tusheng.zhang@manchester.ac.uk}}
\address{ School of Mathematics, University of Manchester, \\
Oxford Road, Manchester M13 9PL, England, UK \\
 \printead{e3}}
 \affiliation{University of Manchester}


\runauthor{Sabbagh and Zhang }


\begin{abstract}
In this paper, we obtain existence and uniqueness of solutions of systems of reflected quasilinear stochastic partial differential equations in a convex domain  $D$. The method is based on the probabilistic interpretation of the solution through the backward doubly stochastic differential equations. The solution is expressed as a pair $(u,\nu)$ where $u$ is a $\bar{D}$-valued, predictable, $L^2(\R^d)$-continuous process which belongs to a proper Sobolev space and $\nu$ is a vector-valued random signed  measure which  prevents the solution from leaving the domain $D$.
\end{abstract}

\end{frontmatter}

\section{Introduction}

Stochastic partial differential equations (SPDEs) are a powerful tool to model various phenomena from biology, engineering to finance. They can be used, for example, to describe the evolution of action potentials in the brain, or to model interest rates. They appeared also in phase transitions and front propagation in random media, in filtering and stochastic control with partial observations, in pathwise stochastic control theory and mathematical finance, etc.\\

It is well known now that backward stochastic differential equations (BDSDEs in short) give a probabilistic interpretation for the solution of a class of semi-linear PDEs.
By introducing in  standard BSDEs a second nonlinear term driven by an external noise, we obtain Backward Doubly SDEs (BDSDEs in short) \cite{PP} (see also \cite{BM}, \cite{MS1}), which give rise to a  representation of the solutions of SPDEs and provide a powerful tool for probabilistic  numerical  schemes \cite{BLMM}  for such SPDEs. Several generalizations to investigate more general nonlinear SPDEs have been developed following different approaches of the notion of weak solutions, namely,  Sobolev's solutions  \cite{DS, GR, PR}. \\

Given a convex domain $D$ in $\R^k$, we are concerned with the study of weak solutions of systems of  reflected quasilinear SPDEs in the domain $D$. We consider a class of PDEs but perturbed by a nonlinear noise driven by a finite -dimensional Brownian motion. We are looking for the solutions of the reflection problem for systems of  quasilinear SPDEs. The solution will be  a pair $(u, \nu)$, where $\nu$ is a vector-valued random signed measure, $ u \in \mathbf { L } ^ 2 \big ( \Omega \times [0 , T ] ; H ^ 1 ( \mathbb R ^ d)^{\otimes k} \big )$ satisfy the following relations:
\begin{equation}
\begin{small}
\label{SPDE1}
 \left\lbrace
\begin{aligned}
& (i) \; \;    u(t,x) \in \bar D , \quad d\mathbb{P}\otimes dt\otimes dx - \mbox{a.e.},  \\
& (ii)\;\; du (t,x) + \Big(\partial_i \big[ a_{ij}(t,x)\partial_j  u(t,x)+g_i(t,x,u (t,x),\nabla u(t,x)) \big] +f(t,x,u(t,x),\nabla u(t,x))\Big)dt\\
&\qquad +  h(t,x,u(t,x),\nabla u(t,x))\cdot d\W_t  = -\nu(dt,dx), a.s. \\
& (iii)\; \; \nu(u\notin \partial  D)=0 , a.s.,\\
 & (iv)  \; \;  u(T,x) = \Phi(x), \quad dx-\mbox{a.e.}.
\end{aligned}
\right.
\end{small}
\end{equation}
where  $a$ is a time-dependent symmetric uniformly elliptic measurable matrix, $f,$
$h$, $g$ are  non-linear measurable functions and Lipschitz in $(y,z)$.
The differential term with $d\W_t$
refers to the backward stochastic integral with respect to a $l$-dimensional Brownian motion on
$\big(\Omega, \mathcal{F},\mathbb{P}, (W_t)_{t\geq 0} \big)$.
\vskip 0.3cm
Real-valued reflected SPDEs driven by space-time white noise was studied in \cite{DP1}, \cite{NP}, \cite{XZ}. Systems of reflected SPDEs in a convex domain were considered in \cite{Z-1}. However, in these literature the gradient of the solution did not appear in the equation.
\vskip 0.3cm
In the one dimensional case, Matoussi and Stoica \cite{MS10} have proved an existence and uniqueness result for the obstacle problem of quasilinear stochastic PDE.
The method is based on the probabilistic interpretation of the solution by using the backward doubly stochastic differential equation (BDSDE in short). They have also proved that the solution is a pair $(u,\nu)$ where $u$ is a predictable continuous process which takes values in a proper Sobolev space and $\nu$ is a random regular measure satisfying the minimal Skohorod condition. In particular, they gave for the regular measure $\nu$ a probabilistic interpretation in terms of the continuous increasing process $K$ where $(Y, Z, K)$ is the solution of a reflected generalized BDSDE. Their method relies on the comparison theorem for SPDEs which is very much one dimensional.
\vskip 0.3cm
The system of reflected semilinear (the case $g=0$)  SPDEs was studied in \cite{MSZ}. Our approach is similar to that in \cite{MSZ}. However, additional difficulties arise when adding  the divergence term $div(g(t,u(t,x), \nabla u(t,x)))$ to the equation.
An essential ingredient to deal with this quasilinear part is the probabilistic representation of the divergence term obtained in \cite{S} in terms of the forward and backward stochastic integrals.
\vskip 0.3cm
We will use the penalization method. To prove the convergence of the solutions of the approximating equations, we appeal to the associated backward doubly stochastic differential equations.
 Indeed, a probabilistic method based on reflected BDSDEs and stochastic flow technics are investigated in our context (see e.g \cite{BM,BCKF} for more details in these technics). The key element  is to use the inversion of stochastic  flow which transforms the variational formulation of the SPDEs to the associated BDSDEs. Thus it plays the same role as It\^o's formula in the case of the classical solution of SPDEs.
We need also to establish a number of   a priori estimates using  an  extension of Ito's formula for the solution of the system of generalized BDSDEs involving forward and backward stochastic integrals.
\vskip 0.3cm
The paper is organized as follows:
We introduce in Section \ref{section:SPDE} several notations and hypothesis that will be used throughout the paper. Then, a weak formulation for the system of quasilinear SPDEs is given in Definition \ref{o-pde}. The main results of this paper are presented  in Section \ref{existenceSPDE:section}. Indeed, the existence and uniqueness result of the weak solution for quasilinear RSPDEs are established by using a penalization method.  A probabilistic representation of this solution is proven via the solution of generalized Markovian RBDSDEs.
In the Appendix, technical lemmas for the existence of the solution of the reflected BDSDEs are given.
\section{Weak solution of quasilinear SPDE in a convex domain}
\label{section:SPDE}
The euclidean norm of a vector $x\in\R^k$ will be denoted by $|x|$, and for a $k\times k$ matrix $A$, we define $\|A\|=\sqrt{Tr AA^\top}$. In what folllows let us fix a positive number $T>0$.\\
Let $(\Omega, \Fc,\P)$ be a probability product space, and let $\{W_s, 0\leq s\leq T\}$ and $\{B_s, 0\leq s\leq T\}$ be two mutually independent standard Brownian motion processes, with values respectively in $\R^l$ and in $\R^d$.
For each $t\in[0,T]$, we define
$$\Fc_t:=\Fc_t^B\vee\Fc_{t,T}^W\vee \Nc$$
where $\Fc_t^B=\sigma\{B_r, 0\leq r\leq t\}$, $\Fc_{t,T}^W=\sigma\{W_r-W_t, t\leq r\leq T\}$ and $\Nc$ the class of $\P$ null sets of $\Fc$.
Note that the collection $\{\Fc_t, t\in[0,T]\}$ is neither increasing nor decreasing, and it does not constitute a filtration.
\subsection{Transformation of the equation}
We note that we can reduce the study of our problem \eqref{SPDE1} using the transformation given in Matoussi and Stoica \cite{MS10} (Remark 1, p. 1157). Indeed, we denote by $L=\sum_{i,j} \partial_i a_{ij}\partial_j$ the elliptic operator such that
$$\lambda |\xi|^2\leq \sum_{i,j} a_{ij}(x) \xi^i\xi^j\leq\Lambda |\xi|^2.$$
then the time change $t\rightarrow \Frac{1}{2 \Lambda} t^\prime$ yields to one correspondence between the solutions $u$ of the equation
\begin{align}
du (t,x) & + \big[ L u(t,x)+ div (g(t,x,u (t,x),\nabla u(t,x))) +f(t,x,u(t,x),\nabla u(t,x))\big]dt \nonumber \\
 &+ h(t,x,u(t,x),\nabla u(t,x))\cdot d\W_t  = 0\, ,
\end{align}
over $[0,T]$ and the solutions $\widehat{u}(t,.)=u(\Frac{1}{2 \Lambda} t,.)$ satisfying the equation
\begin{align}
d\widehat{u}(t,x) & + \big[\Frac{1}{2}\Delta \widehat{u}(t,x) + div (\widehat{g}(t,x,\widehat{u} (t,x),\nabla \widehat{u}(t,x))) +\widehat{f}(t,x,\widehat{u}(t,x),\nabla \widehat{u}(t,x))\big]dt\nonumber \\
 &+ \widehat{h}(t,x,\widehat{u}(t,x),\nabla \widehat{u}(t,x))\cdot d\overleftarrow{\widehat{W}}_t  = 0\, ,
\end{align}
over the interval $[0,2 \Lambda T]$, with the transformed coefficients
$$\widehat{f}(t,x,y,z):=\frac{1}{2\Lambda}f\left(\frac{1}{2\Lambda}t,x,y,z\right)\quad,\quad \widehat{h}(t,x,y,z):=\frac{1}{(2\Lambda)^{1/2}}h\left(\frac{1}{2\Lambda}t,x,y,z\right)$$
$$ \widehat{g}(t, x,y,z):=\frac{1}{2\Lambda}\left( g( \frac{1}{2\Lambda}t,x,y,z) +\gamma \left( x\right) z\right)\quad, \quad \gamma=\Lambda I - a.$$

Therefore, from now on, we focus our study on solving a system of reflected quasilinear stochastic PDEs of the form:
\begin{align}\label{2.1}
du(t,x) & + \big[\Frac{1}{2}\Delta u(t,x) + div (g(t,x,u (t,x),\nabla u(t,x))) +f(t,x,u(t,x),\nabla u(t,x))\big]dt\nonumber \\
 &+ h(t,x,u(t,x),\nabla u(t,x))\cdot d\W_t  = 0.
\end{align}

\vspace{0.5cm}
\noindent  Our main interest is the study of  weak solutions  to the reflection problem for  multidimensional SPDEs  in a  convex domain $ D$ in $ \R^k$. We consider the  solution of system of  refected quasilinear SPDEs \eqref{SPDE1}  as a  pair $ (u, \nu)$, where $ \nu$ is a vector-valued random signed  measure and $ u \in \mathbf{L}^2 \big(\Omega \times [0,T]; H^1 (\mathbb R^d)^{\otimes k}\big)$ satisfies the following relations:
\small
\begin{equation}
\label{RSPDE1}
 \left\lbrace
\begin{aligned}
& (i) \; \;    u(t,x) \in \bar D , \quad d\mathbb{P}\otimes dt\otimes dx - \mbox{a.e.},  \\
& (ii)\;\; du(t,x) + [\frac{1}{2}\Delta u(t,x) +f(t,x,u(t,x),\nabla u(t,x))+div(g(t,x,u(t,x),\nabla u(t,x))) ]dt\\
&\qquad +  h(t,x,u(t,x),\nabla u(t,x))\cdot d\W_t  = -\nu(dt,dx), a.s. \\
& (iii)\; \; \nu(u\notin \partial  D)=0 , a.s.,\\
 & (iv)  \; \;  u(T,x) = \Phi(x), \quad dx-\mbox{a.e.}.
\end{aligned}
\right.
\end{equation}
The random measure $\nu$  acts only when  the process $u$ reaches the boundary of  the domain $ D$.
 The rigorous  sense of the relation $(iii)$ will be based on the  probabilistic representation of the measure $\nu$  in terms of the
  bounded variation process $K$, a component of the associated  solution of the reflected  BDSDE in the domain $D$.
\subsection{Notations and Hypothesis}
Let us first introduce some functional spaces:\\
- $C^n_{l,b}(\R^p,\R^q)$ the set of $C^n$-functions which grow at most linearly at infinity and whose partial derivatives of order less than or equal to $n$ are bounded.\\
- $\mathbf{L}^2\left( \mathbb{{R}}^d\rightarrow \mathbb{{R}}^k \right) $  is the usual $L^2$ space with the inner product,
$$ \left( u,v\right)=\Int_{\mathbb{R}^d}<u\left( x\right), v\left(
x\right)>dx,\;\left\| u\right\| _2=\left(
\Int_{\mathbb{R}^d}|u|^2\left( x\right) dx\right) ^{\frac
12}. $$
Here $<, >$ stands for the scalar product in Euclidean spaces.
\vspace{0.2cm}
Our evolution problem will be considered over a fixed time interval
$[0,T]$ and the norm for an element of $\mathbf{L}^2\left(
[0,T] \times \mathbb{{R}}^d\rightarrow \mathbb{{R}}^k\right) $ will be denoted by
$$\left\| u\right\| _{2,2}=\left(\Int_0^T  \Int_{\mathbb{R}^d} |u (t,x)|^2 dx dt \right)^{\frac 12}. $$
We will write  $\mathbf{L}^2\left( [0,T] \times
\mathbb{{R}}^d\right)$ for $\mathbf{L}^2\left(
[0,T] \times \mathbb{{R}}^d\rightarrow \mathbb{{R}}\right)$.
We introduce the following hypotheses :
\begin{Assumption}\label{assxi}
$\Phi:\R^d\rightarrow\R^k$ belongs to $\mathbf{L}^4(\R^d)$ and $\Phi(x)\in\bar{D}~~ a.e. ~\forall x\in\R^d$;
\end{Assumption}
\begin{Assumption} \label{assgener}
\begin{itemize}
\item[\rm{(i)}] $f:[0,T]\times \R^d\times \R^k\times \R^{k\times
d}\rightarrow\R^k$, $h:[0,T]\times \R^d\times \R^k\times \R^{k\times
d}\rightarrow\R^{k\times l}$ and $g:[0,T]\times\R^d\times\R^k\times\R^{k\times d}\rightarrow\R^{k\times d}$  are measurable in $(t,x,y,z)$ and satisfy:
\begin{itemize}
\item $|f(t,x,y,z)|\leq  f^0(t,x).$
\item $|h(t,x,y,z)|\leq  h^0(t,x).$
\item  $|g(t,x,y,z)|\leq  g^0(t,x).$
\end{itemize}
where $ f^0, h^0 \text{and}\,\,   g^0$ are bounded and belong to $ \mathbf{L}^2\left( [0,T] \times
\mathbb{{R}}^d\right). $
 \item [\rm{(ii)}] There exist constants $c>0$,\,$0<\alpha<1$ and $0<\beta<1$ such that for any $(t,x)\in[0,T]\times\R^d~;~(y_1,z_1),(y_2,z_2)\in\R^k\times\R^{k\times d}$
\b*
|f(t,x,y_1,z_1)-f(t,x,y_2,z_2)| &\leq & c\big(|y_1-y_2|+\|z_1-z_2\|\big)\\
\|h(t,x,y_1,z_1)-h(t,x,y_2,z_2)\| &\leq & c|y_1-y_2|+\beta \|z_1-z_2\|\\
\|g(t,x,y_1,z_1)-g(t,x,y_2,z_2)\| &\leq & c|y_1-y_2|+\alpha \|z_1-z_2\|.
\e*
\item [\rm{(iii)}]The contract property: $\alpha+\Frac{\beta^2}{2}<\Frac{1}{2}$.
\end{itemize}
\end{Assumption}
To avoid technical complications, through the paper we assume $D$ is a regular  (i.e. a convex domain with class $C^2$ boundary) and $0\in D$.
\vskip 0.3cm
Since the domain $D$ is convex we need to  recall some properties that we will use later.  Let $\partial D$ denote the boundary of $D$ and  $\pi(x)$ the  orthogonal projection of $x\in\R^k$ on the closure $\bar{D}$. We have the following properties:
\be
(x'-x)^\top(x-\pi(x))\leq 0,  ~~ \forall x\in\R^d, ~ \forall x'\in \bar{D}\label{prop1}
\ee
\be
(x'-x)^\top(x-\pi(x))\leq (x'-\pi(x'))^\top(x-\pi(x)),  ~~ \forall x, x'\in\R^k\label{prop2}
\ee
\be
\exists \gamma > 0,\, \mbox{such that}\,\, x^\top (x-\pi(x))\geq \gamma|x-\pi(x)|, ~~ \forall x\in\R^k. \label{prop3}
\ee

\noindent One can find all these results in Menaldi \cite{M}, page 737.
%

\subsection{The measures $\P^m$}
The operator $ \partial_t + \frac{1}{2} \Delta $, which represents the main linear part in the equation \eqref{RSPDE1}, is associated with the Bownian motion in $\mathbb{R}^d$.   The sample space of the Brownian motion is $ \Omega' = \mathcal{C }\left([0, \infty ); \mathbb{R}^d \right)$, the canonical process $(B_t)_{t \geq 0}$ is defined by $ B_t (\omega) = \omega (t)$, for any $ \omega \in \Omega'$, $t \geq 0$ and  the shift operator, $ \theta_t  \, : \,  \Omega' \longrightarrow  \Omega'$, is defined by $ \theta_t (\omega) (s) = \omega (t+s)$, for any $s \geq 0$ and $ t \geq 0$. The canonical filtration $ \mathcal{F}_t^B = \sigma \left( B_s; s \leq t \right)$ is completed by the standard procedure with respect to the probability measures produced by the transition function
$$
P_t (x, dy) = q_t (x-y) dy, \quad t >0, \quad x \in \mathbb{R}^d,
$$
where $ q_t (x) = \left(2\pi t\right)^{- \frac{d}{2}} \exp \left( - |x|^2/2t \right)$ is the Gaussian density. Thus we get a continuous Hunt process $\left(\Omega', B_t, \theta_t, \mathcal{F}, \mathcal{F}^B_t, \mathbb{P}^x \right)$. We shall also use the backward filtration of the future events $ \mathcal{F}'_t = \sigma \left(B_s; \; \,  s \geq t \right)$ for $t\geq 0$. $\mathbb{P}^0$ is the Wiener measure, which is supported by the set $ \Omega'_0 = \{ \omega \in \Omega', \; \, \omega(0) =0 \}$. We also set $ \Pi_0 (\omega) (t) = \omega (t) - \omega (0),\,  t \geq 0$, which defines a map $ \Pi_0 \, : \, \Omega' \rightarrow \Omega'_0$. Then  $\Pi = (B_0, \Pi_0 ) \, : \,  \Omega' \rightarrow \mathbb{R}^d \times \Omega'_0$ is a bijection. For each probability measure on $\mathbb{R}^d$, the probability $\mathbb{P}^{\mu}$ of the Brownian motion started with the initial distribution $\mu$ is given by $$ \mathbb{P}^{\mu} = \Pi^{-1} \left(\mu \otimes \mathbb{P}^0 \right).$$
In particular, for the Lebesgue measure in $\mathbb{R}^d$, which we denote by $ m = dx$, we have
$$ \mathbb{P}^{m} = \Pi^{-1} \left(dx\otimes \mathbb{P}^0 \right).$$
We recall that  $\{B_{t,s}(x), t\leq s\leq T\}$ is the diffusion process starting from $x$ at time $t$ and is given by
 \begin{equation}\label{sde}
  B_{t,s}(x)=x+(B_s-B_t).
\end{equation}
Moreover the inverse of the flow satisfies the
following backward SDE
\begin{equation}\label{inverse:flow}
\begin{split}
B_{t,s}^{-1}(y) &  = y   - (B_s-B_t) .
\end{split}
\end{equation}
for any  $t<s$.
\vskip 0.3cm
Notice that $\{B_{t-r}, \mathcal{F}'_{t-r}, r\in [0, t]\}$ is a backward local martingale under $\P^m$.
Let $L(\cdot,\cdot): [0, \infty )\times \R^d \rightarrow \R^d$ be a measurable function such that $L\in \mathbf{L}^2\left(
[0,T] \times \mathbb{{R}}^d\rightarrow \mathbb{{R}}^d\right) $ for any $T>0$. Finally we recall the forward and backward stochastic integral defined in \cite{MS10} under the measure $\P^m$.
$$\Int_s^t L(r,B_{r})\ast dB_r=\Int_s^t <L(r,B_{r}), dB_r>+\Int_s^t <L(r,B_{r}), d\overleftarrow{B}_r>. $$
When $L$ is smooth, one has
\begin{equation}\label{forward-backward}
\Int_s^t L(r,B_{r})\ast dB_r=-2\Int_s^t div(L(r,\cdot))(B_r)dr.
\end{equation}
We refer the reader to \cite{MS10} for more details.
\subsection{Weak formulation for a solution of stochastic  PDEs}
\label{definition:solution}
The space of test functions which we employ in the definition of
weak solutions of the evolution equations  \eqref{SPDE1} is $
\mathcal{D}_T  = \left [\mathcal{C}^{\infty} (\left[0,T]\right) \otimes
\mathcal{C}_c^{\infty} \left(\mathbb{R}^d\right)\right ]^{\otimes k}$, where
\begin{itemize}
\item $\mathcal{C}^{\infty} \left([0,T]\right)$ denotes the space of real
functions which can be extended as infinitely differentiable functions
in the neighborhood of $[0,T]$,
\item $\mathcal{C}_c^{\infty}\left(\mathbb{R}^d\right)$ is the space of
infinitely differentiable functions with compact supports in
$\mathbb{R}^d$.
\end{itemize}
Another space that we use is the first order Sobolev space $ H^1 (\R^d)=H^1_0 (\R^d)$. Its natural scalar product and norm are
$$(u,v)_{H^1 (\R^d)}= (u,v)+(\nabla u,\nabla v), \quad \|u\|_{H^1 (\R^d)}= \left(\|u\|_2^2+\|\nabla u\|_2^2\right)^{1/2},$$
where $\nabla$ stands for the gradient. Here, the derivative is defined in the weak sense (Sobolev sense).\\
We denote by  $ {\mathcal H}_T\subset \mathbf{L}^2 ([0,T]; H ^1 (\R^d)^{\otimes k} )$, $\P$-a.s., the space of  $ \Fc_{t,T}^W$-progressively measurable  processes  $(u_t ) $ such that $t\mapsto u_t=u(t,.)$ is continuous in $\mathbf{L}^2(\R^d\rightarrow \R^k)$ endowed with the norm
$$\begin{array}{ll}
\|u\|_{{\mathcal H}_T}^2=
 \E \,  \Big[\underset{ 0 \leq s \leq T}{\Sup} \|u_s \|_2^2 +   \Int_{ \mathbb{R}^d} \Int_0^T  |\nabla
u_s (x)|^2 ds dx \Big].
\end{array}
$$
\begin{Definition}[{\textbf{Weak solution of quasilinear SPDE without reflection}}]
We say that $ u \in \mathcal{H}_T $ is a weak solution of SPDE  $(
\ref{2.1}) $ if the following
relation holds, for each $\varphi \in \mathcal{D}_T ,$
\begin{align}\label{wspde1}
\nonumber &\Int_{t}^{T}\!\!\Int_{\mathbb{R}^{d}}\!\big[\langle u(s,x), \partial _{s}\varphi(s,x)\rangle+\Frac{1}{2}\langle \nabla u(s,x), \nabla \varphi(s,x)\rangle\big]dxds+\Int_{\mathbb{R}^{d}}\!\!\big[\langle u(t,x ), \varphi (t,x
)\rangle-\langle \Phi(x ), \varphi (T,x)\rangle\big]dx\\
\nonumber &=\Int_{t}^{T}\!\!\Int_{\mathbb{R}^{d}}\!\!\big[\langle (s,x ,u(s,x),\nabla u(s,x)), \varphi(s,x)\rangle-\langle g(s,x ,u(s,x),\nabla u(s,x)), \nabla\varphi(s,x)\rangle\big]dxds\\
&+\Int_{t}^{T}\!\!\Int_{\mathbb{R}^{d}}\!\!\langle h(s,x ,u(s,x),\nabla u(s,x)), \varphi(s,x) \rangle dxd\W_s.
\end{align}
We denote by $ u:=\mathcal{U }(\Phi, f,g,h)$ the solution of SPDEs with data $(\Phi,f,g,h)$.
\end{Definition}
\vspace{0.2cm}
\noindent The existence and uniqueness  of weak solution  for SPDEs \eqref{wspde1} is ensured by Theorem 8 in Denis and Stoica \cite{DS}.\vspace{0.3cm}
\noindent\\
\newpage
We  now precise  the definition  of weak solutions for the reflected quasilinear SPDE (\ref{RSPDE1}):
\begin{Definition}[{\textbf{Weak solution of quasilinear RSPDE}}]
\label{o-pde}We say that $(u,\nu ):= (u^i,\nu^i )_{1\leq i\leq k}$ is a weak solution of the reflected SPDE (\ref{RSPDE1}) associated to $(\Phi,f,g,h)$, if for each $1\leq i\leq k$
\begin{itemize}
\item[(i)]$\left\| u\right\|_{{\mathcal H}_T} <\infty $, $u_t(x)\in \bar{D}, dx\otimes dt\otimes d\P~a.e.$, and  $u(T,x)=\Phi(x)$.
\item[(ii)] $\nu^i $ is a signed \textit{random measure} on $[0,T]\times\R^d$ such that:
\begin{itemize}
\item[a)] $\nu^i $ is adapted in the sense that for any measurable function $\psi:[0,T]\times \R^d\longrightarrow\R$ and for each $s\in[t,T]$,$\Int_{s}^{T}\!\Int_{\mathbb{R}^{d}}\!\psi (r,x)\nu^i(dr,dx)$ is $\Fc_{s,T}^W$-measurable.
\item[b)] $\E\big[\Int_{0}^{T}\Int_{\mathbb{R}^{d}} |\nu^i| (dt,dx)\big]<\infty.$
\end{itemize}
\item[(iii)] for every $\varphi \in \mathcal D_T$%
\begin{align}\label{OPDE}
\nonumber &\Int_{t}^{T}\!\!\Int_{\mathbb{R}^{d}}\!\big[\langle u(s,x), \partial _{s}\varphi(s,x)\rangle+\Frac{1}{2}\langle\nabla u(s,x), \nabla \varphi(s,x)\rangle\big]dxds+\Int_{\mathbb{R}^{d}}\!\!\big[\langle u(t,x ), \varphi (t,x
)\rangle-\langle\Phi(x ), \varphi (T,x)\rangle\big]dx\\
\nonumber &=\Int_{t}^{T}\!\!\Int_{\mathbb{R}^{d}}\!\!\big[\langle f(s,x ,u(s,x),\nabla u(s,x)), \varphi(s,x)\rangle-\langle g(s,x ,u(s,x),\nabla u(s,x)), \nabla\varphi(s,x)\rangle\big]dxds\\
&+\Int_{t}^{T}\!\!\Int_{\mathbb{R}^{d}}\!\!\langle h(s,x ,u(s,x),\nabla u(s,x)), \varphi(s,x)\rangle dx d\W_s
+\Int_{t}^{T}\!\!\Int_{\mathbb{R}^{d}}\!\langle\varphi (s,x)1_{\{u\in \partial D\}}(s,x), \nu(ds,dx)\rangle.
\end{align}
\end{itemize}
\end{Definition}

\section{Existence and uniqueness of the system of reflected quasilinear SPDEs}
\label{existenceSPDE:section}
In this section, we will establish the existence and uniqueness result of the weak solution for quasilinear RSPDEs \eqref{RSPDE1} by using a penalization method. As a byproduct, we also obtain a probabilistic representation of this solution via the solution of generalized Markovian RBDSDEs.
The first main result of this section is the following:
\begin{Theorem}
\label{existence:RSPDE}
Let  Assumptions \ref{assxi}, \ref{assgener} hold. Then
there exists a unique weak solution $(u,\nu)$ of  the reflected SPDE (\ref{RSPDE1}) associated to $(\Phi,f,g,h)$ that satisfies
$ u (t, x) := Y_t^{t,x}$,   $dt\otimes d\P\otimes d\P^m-a.e.$,  and
\be\label{con-pre}
 Y_{s}^{t,x}=u(s,B_{t,s}(x)), \quad \quad Z_{s}^{t,x}=(\nabla u)(s,B_{t,s}(x)), \quad ds\otimes
d\P\otimes-a.e.,
\ee
where $(Y_{s}^{t,x},Z_{s}^{t,x},K_{s}^{t,x})_{t\leq s\leq T}$ is the
solution of the Markovian RBDSDE
\begin{equation}
\label{RBDSDE}
 \left\lbrace
\begin{aligned}
&(i)~ Y_{s}^{t,x}
 =\Phi(B_{t,T}(x))+
\Int_{s}^{T}f(r,B_{t,r}(x),Y_{r}^{t,x},Z_{r}^{t,x})dr+\Int_{s}^{T}h(r,B_{t,r}(x),Y_{r}^{t,x},Z_{r}^{t,x})d\W_r+K_{T}^{t,x}-K_{s}^{t,x}\\
&\hspace{2.5cm}
 +\frac{1}{2}\Int_t^T g(r,B_{t,r}(x),Y_{r}^{t,x},Z_{r}^{t,x})\ast dB_r-\Int_{s}^{T}Z_{r}^{t,x}dB_{r},\;
\P\otimes \P^m\text{-}a.e. , \; \forall \,  s \in [t,T]  \\
& (ii)~ Y_{s}^{t,x} \in \bar{D} \, \, \quad \P\otimes \P^m\text{-}a.e.\\
& (iii) \Int_0^T (Y_{s}^{t,x}-v_s(B_{t,s}(x)))^* dK_{s}^{t,x}\leq 0., \,\P\otimes \P^m\text{-}a.e., \\
&~ \text{for any continuous }\, \Fc_t -\text{random function}\, \quad v \, : \,[0,T] \times \Omega \times \mathbb R^d \longrightarrow \,  \bar{D}.
\end{aligned}
\right.
\end{equation}
Furthermore, for every measurable bounded and positive functions $\varphi $ and $\psi $,
\begin{align}
\Int_{\mathbb{R}^{d}}\Int_{t}^{T}\varphi (s,B^{-1}_{t,s}(x))\psi (s,x)1_{\{u\in \partial D\}}(s,x)\nu^i (ds,dx)=\Int_{\mathbb{R}^{d}}\Int_{t}^{T}\varphi (s,x)\psi (s,B_{t,s}(x))dK_{s}^{i}dx\text{, a.s..}
\label{con-k}
\end{align}
\end{Theorem}
\subsection{Proof of existence}
The existence of a solution will be proved by a penalization method. For $n\in\N$, we consider the penalized system of quasilinear SPDE:
\be\label{SPDEpen}
\left\lbrace
\begin{aligned}
du^n(s,x) &+ [\frac{1}{2}\Delta u^n(s,x) +f(s,x,u^n (s,x),\nabla u^n(s,x))+div\big(g(s,x,u^n (s,x),\nabla u^n(t,x))\big) ] ds\\
&+  h(s,x,u^n(s,x),\nabla u^n(s,x))\cdot d\W_s -n(u^n(s,x)-\pi (u^n(s,x)))ds=0,\\
u^n(T,x)&=\Phi(x)
\end{aligned}
\right.
\ee
From Denis and Stoica \cite{DS} (Theorem 8), we know that the above equation admits a unique weak solution $\mathcal{U }(\Phi,f^{n},g,h)$ (\eqref{wspde1}), with $
f^{n}(t,x,y)=f(t,x,y,z)-n(y-\pi(y))$, i.e. for every $\varphi \in\mathcal{D}_T$.
\begin{align}\label{o-equa1}
 \nonumber\Int_{t}^{T}&\big[(u^{n}(s,\cdot),\partial
_{s}\varphi(s,\cdot) )+\Frac{1}{2}(\nabla u^{n}(s,\cdot),\nabla\varphi(s,\cdot))\big]ds  +(u^{n}(t,\cdot ),\varphi
(t,\cdot ))-(\Phi(\cdot ),\varphi (T,\cdot))\\
\nonumber&=\Int_{t}^{T}\big[(f^n(s,\cdot,u^{n}(s,\cdot),\nabla u^{n}(s,\cdot)),\varphi(s,\cdot))+(g(s,\cdot,u^{n}(s,\cdot),\nabla
u^{n}(s,\cdot)),\nabla\varphi(s,\cdot))\big]ds\\
&+\Int_{t}^{T}(h(s,\cdot,u^{n}(s,\cdot),\nabla u^{n}(s,\cdot)),\varphi(s,\cdot))d\W_s.
\end{align}
 We are going to show that $(u^n)_{n\geq 1}$ is a Cauchy sequence in $\Hc_T$ with the help of backward doubly stochastic differential equations. Denote \begin{equation}
\begin{split}
\label{representation}
Y_s^{n,t,x}=u^n(s,B_{t,s}(x))\quad , \quad Z^{n,t,x}_s=\nabla u^n(s,B_{t,s}(x))\\
K_s^{n,t,x}=-n\Int_0^s\big[u^n(r,B_{t,r}(x))-\pi(u^n(r,B_{t,r}(x)))\big]dr.
\end{split}
\end{equation}
 By following the representation in  Matoussi and Stoica \cite{MS10} (Theorem 1 p.1148),  we see that $(Y^{n,t,x},Z^{n,t,x})$ solves the following doubly backward stochastic differential equations $\P\otimes \P^m$-a.e.:
\begin{eqnarray}\label{BDSDEpen}
Y_{s}^{n,t,x}&=&\Phi(B_{t,T}(x))+\Int_{s}^{T}f(r,B_{t,r}(x),Y_{r}^{n,t,x},Z_{r}^{n,t,x})dr+\Int_{s}^{T}h(r,B_{t,r}(x),Y_{r}^{n,t,x},Z_{r}^{n,t,x})d\W_r\nonumber\\
&&+K_T^{n,t,x}-K_s^{n,t,x}+\frac{1}{2}\Int_s^T g(r,B_{t,r}(x),Y_{r}^{n,t,x},Z_{r}^{n,t,x})\ast dB_r-\Int_{s}^{T}Z_{r}^{n,t,x}dB_{r}.\nonumber\\
& &
\end{eqnarray}
\begin{Remark}
The subscripts $(t,x)$ will often be dropped for notational simplicity if the context is clear and the notations $B_t=B_{t,s}(x)$ and $B_t^{-1}=B_{t,s}^{-1}(y)$ will be frequently used throughout.
\end{Remark}
\vspace{0.1cm}
\noindent Next we will prove that $(Y^{n},Z^{n}), n\geq 1$ is a Cauchy sequence. To this end, we need to prepare a number of preliminary results.
\noindent We start with the following lemma:
\begin{Lemma} \label{lem}There exists a constant $C>0$ such that
\be
\forall n\in\N \qquad  \E\E^m\big[\Int_0^T d^2(Y_s^n,D)ds\big]\leq \Frac{C}{n}.
\ee
\end{Lemma}
\begin{proof}
We apply the double stochastic It\^o's formula extended in Matoussi and Stoica (Corollay 1 and Remark 2 in \cite{MS10} p.1158) to $\rho(u^n(t,B_t))=\rho(Y_t^n)=d^2(Y_t^n,D)=|Y_t^n-\pi(Y_t^n)|^2$ to obtain
\be\label{2.3}\begin{split}
\rho(Y_t^n)&+ \Frac{1}{2}\Int_t^T trace[Z_s^nZ_s^{n \top}Hess \rho(Y_s^n)]ds = \rho(\Phi(B_{T}))
+ \Int_t^T (\nabla \rho(Y_s^n))^\top f(s,B_s,Y_s^n,Z_s^n)ds\\&- \Int_t^T (\nabla \rho(Y_s^n))^\top Z_s^n dB_s + \Int_t^T (\nabla \rho(Y_s^n))^\top h(s,B_s,Y_s^n,Z_s^n)d\W_s\\
& + \Frac{1}{2}\Int_t^T trace[(hh^\top)(s,B_s,Y_s^n,Z_s^n)Hess \rho(Y_s^n)]ds - 2n \Int_t^T (Y_s^n- \pi(Y_s^n))^\top(Y_s^n-\pi(Y_s^n))ds\\
&+\frac{1}{2}\sum_{i=1}^k\Int_t^T\frac{\partial \rho}{\partial y_i}(Y_s^n)g^i(s,B_s,Y_s^n,Z_s^n)\ast dB_s-\sum_{i=1}^k\Int_t^T\langle\nabla [\frac{\partial \rho}{\partial y_i}(u^n(s,B_s))], g^i(s,B_s,Y_s^n,Z_s^n)\rangle ds,
\label{itorho}
\end{split}
\ee
where $\nabla$ is taken for the argument $x\in \R^d$. Now, for the last term we have
\begin{align}\label{2.4}
-\sum_{i=1}^k & \Int_t^T\langle\nabla [\frac{\partial \rho}{\partial y_i}(u^n(s,B_s))], g^i(s,Y_s^n,Z_s^n)\rangle ds=\!-\sum_{i=1}^k\Int_t^T\sum_{l=1}^d\frac{\partial }{\partial x_l}\big[\frac{\partial \rho}{\partial y_i}(u^n(s,B_s))\big] g^{il}(s,Y_s^n,Z_s^n)ds\nonumber\\
&=\!-\sum_{i=1}^k\Int_t^T\sum_{l=1}^d\sum_{j=1}^k\big[\frac{\partial^2 \rho}{\partial y_j\partial y_i}(u^n(s,B_s))\frac{\partial u^{n,j}(B_s)}{\partial x_l}\big] g^{il}(s,Y_s^n,Z_s^n)ds\nonumber\\
&=-\sum_{i=1}^k\Int_t^T\sum_{l=1}^d\sum_{j=1}^k\big[\frac{\partial^2 \rho}{\partial y_j\partial y_i}(Y^n_s)Z^{n,jl}_s\big] g^{il}(s,Y_s^n,Z_s^n)ds\nonumber\\
&=-\sum_{l=1}^d\Int_t^T\langle Hess\rho(Y^n_s)Z^{n,\cdot l}_s, g^{\cdot l}(s,Y_s^n,Z_s^n)\rangle ds,
\end{align}
where $Z^{n,\cdot l}_s, g^{\cdot l}$ stands for the column vector. Noting that
\begin{align*}
|\langle &Hess\rho(Y^n_s)Z^{n,\cdot l}_s, g^{\cdot l}(s,B_s,Y_s^n,Z_s^n)\rangle |\\
&\leq  \langle Hess\rho(Y^n_s)Z^{n,\cdot l}_s, Z^{n,\cdot l}_s\rangle^{\frac{1}{2}}\langle Hess\rho(Y^n_s)g^{\cdot l}(s,B_s,Y_s^n,Z_s^n), g^{\cdot l}(s,B_s,Y_s^n,Z_s^n)\rangle^{\frac{1}{2}},
\end{align*}
it follows that
\begin{align}
\label{2.5}
-\sum_{i=1}^k& \Int_t^T\langle\nabla [\Frac{\partial \rho}{\partial y_i}(u^n(s,B_s))] g^i(s,B_s,Y_s^n,Z_s^n)\rangle ds\nonumber\\
&\leq \Frac{1}{4}\sum_{l=1}^d\Int_t^T\langle Hess\rho(Y^n_s)Z^{n,\cdot l}_s, Z^{n,\cdot l}_s\rangle ds+C\sum_{l=1}^d\Int_t^T\langle Hess\rho(Y^n_s)g^{\cdot l}(s,Y_s^n,Z_s^n),g^{\cdot l}(s,B_s,Y_s^n,Z_s^n)\rangle ds\nonumber\\
&\leq \frac{1}{4}\Int_t^Ttrace[Z_s^nZ_s^{n \top}Hess\rho(Y^n_s)]ds+C\Int_t^Ttrace[gg^\top Hess\rho(Y^n_s)]ds.\nonumber\\
&\leq \frac{1}{4}\Int_t^Ttrace[Z_s^nZ_s^{n \top}Hess\rho(Y^n_s)]ds+C\Int_t^T \|g^0(s,B_s) \|^2ds.\nonumber\\
\end{align}
Since $\Phi(B_{T})\in \bar{D} ~a.s.$, we have that $\rho(\Phi(B_{T}))=0$.
Substituting \eqref{2.5} into \eqref{2.3} and taking into account Assumption \ref{assgener} (i) and the boundedness of the Hessian of $\rho$ we obtain that
\be
\begin{split}
\rho(Y_t^n)&+ \Frac{1}{4}\Int_t^T trace[Z_s^nZ_s^{n \top}Hess \rho(Y_s^n)]ds  + 2 n\Int_t^T d^2(Y_s^n,D)ds\\
&\leq 2 \Int_t^T (\rho(Y_s^n))^{1/2} |f(s,B_s,Y_s^n,Z_s^n)|ds - 2 \Int_t^T (Y_s^n- \pi(Y_s^n))^\top Z_s^n dB_s\\
&+2 \Int_t^T (Y_s^n- \pi(Y_s^n))^\top h(s,B_s,Y_s^n,Z_s^n)d\W_s + C\Int_t^T \|h^0(s,B_s) \|^2ds\\
&+\frac{1}{2}\sum_{i=1}^k\Int_t^T\frac{\partial \rho}{\partial y_i}g^i(s,B_s,Y_s^n,Z_s^n)\ast dB_s+C\Int_t^T \|g^0(s,B_s) \|^2ds .
\end{split}
\ee
Now the inequality $2ab\leq a^2+b^2$ with $a=\sqrt{\Frac{n}{2}\rho(Y_s^n)}$ yields
\b*
(\rho(Y_s^n))^{1/2} |f(s,B_s,Y_s^n,Z_s^n)|&\leq &\Frac{n}{4}\rho(Y_s^n)+\Frac{1}{n} |f(s,B_s,Y_s^n,Z_s^n)|^2\nonumber\\
&\leq &\Frac{n}{4}\rho(Y_s^n)+\Frac{1}{n} |f^0(s,B_s)|^2.
\e*
Then it follows that
\be
\begin{split}
\rho(Y_t^n)&+ \Frac{1}{4}\Int_t^T trace[Z_s^nZ_s^{n \top}Hess \rho(Y_s^n)]ds  + \Frac{3n}{2}\Int_t^T d^2(Y_s^n,D)ds\\
&\leq  2\Int_t^T\Frac{1}{n}|f^0(s,B_s)|^2ds - 2 \Int_t^T (Y_s^n- \pi(Y_s^n))^\top Z_s^n dB_s \\
&+2 \Int_t^T (Y_s^n- \pi(Y_s^n))^\top h(s,B_s,Y_s^n,Z_s^n)d\W_s + C\Int_t^T \|h^0(s,B_s) \|^2ds\\
&+\frac{1}{2}\sum_{i=1}^k\Int_t^T\frac{\partial \rho}{\partial y_i}g^i(s,B_s,Y_s^n,Z_s^n)\ast dB_s+C\Int_t^T \|g^0(s,B_s) \|^2ds .
\label{estY}
\end{split}
\ee
By taking expectation and using  the fact that under the measure $\P^m$ the forward-backward integral $\Int \frac{\partial \rho}{\partial y_i}g^i(s,B_s,Y_s^n,Z_s^n)\ast dB_s$ as well the other stochastic integrals with respect to the Brownian terms have null expectation under $\P\otimes\P^m$, we have for all $0\leq t\leq T$
\be
\begin{split}
\E\E^m[\rho(Y_t^n)]+\Frac{1}{4}\E\E^m[\Int_t^T trace[Z_s^nZ_s^{n \top}Hess \rho(Y_s^n)]ds] + \Frac{3n}{2}\E\E^m[\Int_t^T d^2(Y_s^n,D)ds]\leq  C\big(1+\Frac{1}{n}\big).\\
\label{estdist}
\end{split}
\ee
Hence, we deduce that
$$\forall n\in\N \qquad  \E\E^m\big[\Int_0^T d^2(Y_s^n,D)ds\big]\leq \Frac{C}{n}.$$
\ep
\end{proof}
In order to prove the strong convergence of the sequence $(Y^n,Z^n,K^n)$, we shall need the following result.
\begin{Lemma}
\label{fundamental:lemma}
\be
\E\E^m\Big[\underset{0\leq t \leq T}{\Sup}(d(Y_t^n,D))^4\Big]\underset{n\rightarrow +\infty}{\longrightarrow} 0.
\label{dist}
\ee
\end{Lemma}
\begin{proof}
We denote by $\rho(x)=d^2(x,D)$ and $\varphi(x)= \rho^2(x)$.
By applying the double stochastic It\^o's formula to $\varphi(Y_t^n)=d^4(Y_t^n,D)$, we obtain that
\be\begin{split}\label{2.6}
\varphi(Y_t^n)&+ \Frac{1}{2}\Int_t^T trace[Z_s^nZ_s^{n \top}Hess \varphi(Y_s^n)]ds = \varphi(\Phi(B_{T}))
+ \Int_t^T (\nabla \varphi(Y_s^n))^\top f(s,B_s,Y_s^n,Z_s^n)ds\\&- \Int_t^T (\nabla \varphi(Y_s^n))^\top Z_s^n dB_s + \Int_t^T (\nabla \varphi(Y_s^n))^\top h(s,B_s,Y_s^n,Z_s^n)d\W_s\\
& + \Frac{1}{2}\Int_t^T trace[(hh^\top)(s,B_s,Y_s^n,Z_s^n)Hess \varphi_(Y_s^n)]ds - n \Int_t^T \nabla \varphi(Y_s^n))^\top(Y_s^n-\pi(Y_s^n))ds\\
&+\frac{1}{2}\Int_t^T (\nabla \varphi(Y^n_s))^\top g(s,B_s,Y_s^n,Z_s^n)\ast dB_s-\sum_{i=1}^k\Int_t^T \langle\nabla \frac{\partial \varphi}{\partial y_i}(u^n(s,\cdot)),g^i(s,B_s,u^n(s,\cdot),\nabla u^n(s,\cdot))\rangle(B_s)ds.
\label{ito}
\end{split}
\ee
Using the similar arguments leading to the proof of (\ref{2.5}), we obtain
\begin{align}\label{2.7}
-\sum_{i=1}^k& \Int_t^T\langle\nabla [\frac{\partial \varphi}{\partial y_i}(u^n(s,B_s))], g^i(s,B_s,Y_s^n,Z_s^n)\rangle ds\nonumber\\
&\leq \frac{1}{4}\sum_{l=1}^d\Int_t^T\langle Hess\varphi(Y^n_s)Z^{n,\cdot l}_s, Z^{n,\cdot l}_s\rangle ds
+C\sum_{l=1}^d\Int_t^T\langle Hess\varphi(Y^n_s)g^{\cdot l}(s,B_s,Y_s^n,Z_s^n), g^{\cdot l}(s,B_s,Y_s^n,Z_s^n)\rangle ds\nonumber\\
&\leq\frac{1}{4}\Int_t^Ttrace[Z_s^nZ_s^{n\top}Hess\varphi(Y^n_s)]ds
+C\Int_t^Ttrace[gg^\top Hess\varphi(Y^n_s)]ds.
\end{align}
Since $\Phi(B_{T})\in \bar{D} ~a.s.$, we have that $\varphi(\Phi(B_{T}))=0$ and it is easy to see that \be
\nabla \varphi(x)&=&2\rho(x)\nabla\rho(x)=4\rho(x)(x-\pi(x)) \label{gradphi} \\
Hess \varphi(x)&=&2\nabla\rho(x)\otimes\nabla\rho(x)+2\rho(x) Hess \rho(x)\label{hessphi}.
\ee
Combining (\ref{2.6}), (\ref{2.7}) together  it follows that
\be\label{Ito}
\begin{split}
\varphi(Y_t^n)+ \Frac{1}{4}&\Int_t^T trace[Z_s^nZ_s^{n \top}Hess \varphi(Y_s^n)]ds
\leq  4\Int_t^T (\rho(Y_s^n)(Y_s^n-\pi(Y_s^n))^\top f(s,B_s,Y_s^n,Z_s^n)ds\\&- 4\Int_t^T (\rho(Y_s^n)(Y_s^n-\pi(Y_s^n))^\top Z_s^n dB_s + 4\Int_t^T (\rho(Y_s^n)(Y_s^n-\pi(Y_s^n))^\top h(s,B_s,Y_s^n,Z_s^n)d\W_s\\
&+\frac{1}{2}\Int_t^T (\nabla \varphi)^\top(Y^n_s)g(s,B_s,Y_s^n,Z_s^n)\ast dB_s- 4n \Int_t^T \rho^2(Y_s^n)ds\\
& + \Frac{1}{2}\Int_t^T trace[(hh^\top)(s,B_s,Y_s^n,Z_s^n)Hess \varphi(Y_s^n)]ds +C\Int_t^Ttrace[gg^\top Hess\varphi(Y^n_s)]ds.
\end{split}
\ee
By taking expectation under $\P\otimes\P^m$ we have
\be
\begin{split}\label{estvarphi}
\E\E^m[\varphi(Y_t^n)]&+ \Frac{1}{4}\E\E^m\big[\Int_t^T trace[Z_s^nZ_s^{n \top}Hess \varphi(Y_s^n)]ds\big]+4n \E\E^m\big[\Int_t^T \varphi(Y_s^n)ds\big]\\
& \leq 4\E\E^m\big[\Int_t^T (\rho(Y_s^n)(Y_s^n-\pi(Y_s^n))^\top f(s,B_s,Y_s^n,Z_s^n)ds]\\
 &+ \Frac{1}{2}\E\E^m\Big[\Int_t^T trace[(hh^\top)(s,B_s,Y_s^n,Z_s^n)Hess \varphi(Y_s^n)]ds]\\
 &+C\E\E^m[\Int_t^Ttrace[gg^\top Hess\varphi(Y^n_s)]ds].
\end{split}
\ee
Taking into account the boundedness   of $h$ and $Hess\rho$, we have
\be \label{esth}
\begin{split}
\E\E^m\Big[\Int_t^T trace[(hh^\top)(s,B_s,Y_s^n,Z_s^n)& Hess \varphi(Y_s^n)]ds]\leq 2\E\E^m\Big[\Int_t^T\langle h(s,B_s,Y_s^n,Z_s^n),\nabla\rho(Y_s^n)\rangle^2 ds\Big]\\
& + \E\E^m\Big[\Int_t^T 2\rho(Y_s^n)trace[(hh^\top)(s,B_s,Y_s^n,Z_s^n)Hess \rho(Y_s^n)]ds\Big]\\
&\leq C\E\E^m\Big[\Int_t^T|\nabla\rho(Y_s^n)|^2 ds\Big] + C\E\E^m\Big[\Int_t^T 2\rho(Y_s^n)ds\Big]\\
&\leq C  \E\E^m\Big[\Int_0^T(d(Y_s^n,D))^{2}ds\Big].
\end{split}
\ee
Apply the same argument to obtain
\be\label{2.8}
\E\E^m\Big[\Int_t^Ttrace[(gg^\top)(s,B_s,Y_s^n,Z_s^n)Hess\varphi(Y^n_s)]ds\Big ]
\leq C  \E\E^m\Big[\Int_0^T(d(Y_s^n,D))^{2}ds\Big].
\ee
Now the inequality $2ab\leq a^2+b^2$ with $a=(d(Y_s^n,D))^{2}$ and the boundedness of $f$ yield
\be
\begin{split}\label{estf}
4(d(Y_s^n,D))^{3} |f(s,B_s,Y_s^n,Z_s^n)|&\leq 2 (d(Y_s^n,D))^{4}+ 2 (d(Y_s^n,D))^{2}|f(s,B_s,Y_s^n,Z_s^n)|^2\\
&\leq 2 \varphi(Y_s^n) + 2C(d(Y_s^n,D))^{2} .
\end{split}
\ee
By plugging the estimate (\ref{estf}), (\ref{2.8}) and \eqref{esth} in (\ref{estvarphi}), we obtain thanks to lemma \ref{lem}
\be\label{estphi}
\begin{split}
\E\E^m[\varphi(Y_t^n)]+ \Frac{1}{4}\E\E^m\big[\Int_t^T trace [Z_s^nZ_s^{n \top}& Hess \varphi(Y_s^n)]ds\big]+(4n-2) \E\E^m\big[\Int_t^T \varphi(Y_s^n)ds\big]\\
& \leq C\E\E^m\big[\Int_0^T (d(Y_s^n,D))^{2}ds\big]\leq \Frac{C}{n}.
\end{split}
\ee
Notice also that  Hessian of $ \varphi(Y_s^n)$ is a positive definite matrix since $ \varphi$ is a convex function, so we get that $ \E\big[\Int_t^T trace [Z_s^nZ_s^{n *} Hess \varphi(Y_s^n)]ds\big] \geq 0$ and consequently
\be\label{unifY}
\underset{0\leq t\leq T}{\Sup}\E\E^m[\varphi(Y_t^n)]\leq \Frac{C}{n}.
\ee
Moreover, we can deduce from \eqref{estphi} that, for every $t\in[0,T]$
\be\label{estZ}
 \E\E^m\big[\Int_t^T trace [Z_s^nZ_s^{n \top} Hess \varphi(Y_s^n)]ds\big]\longrightarrow 0, \, \text{as}\, n\rightarrow \infty.
\ee
On the other hand, taking the supremum over $t$ in the equation \eqref{Ito} and by Burkholder-Davis-Gundy's inequlity and the previous calculations it follows that
\begin{align}
\label{uniformestimate}
\begin{split}
\E\E^m[\underset{0\leq t\leq T}{\Sup}\varphi(Y_t^n)]&\leq C \E\E^m[\Int_0^T\varphi(Y_s^n) ds]+C\E\E^m\Big[\Int_0^T (d(Y_s^n,D))^{2}ds\Big]\\
&+C\E\E^m\Big[\underset{0\leq t\leq T}{\Sup}\Int_t^T (\rho(Y_s^n)\nabla\rho(Y_s^n))^\top Z_s^n dB_s\Big]\\
&+ C \E\E^m\Big[\underset{0\leq t\leq T}{\Sup}\Int_t^T (\rho(Y_s^n)\nabla\rho(Y_s^n))^\top h_s(Y_s^n,Z_s^n)d\W_s\Big]\\
&+ C \E\E^m\Big[\underset{0\leq t\leq T}{\Sup}\Int_t^T (\nabla\varphi(Y_s^n))^\top g_s(Y_s^n,Z_s^n)\ast d\B_s\Big]\\
&\leq C \E\E^m[\Int_0^T\varphi(Y_s^n) ds]+C\E\Big[\Int_0^T (d(Y_s^n,D))^{2}ds\Big]\\&+C\E\E^m\Big[\Big(\Int_0^T (\rho(Y_s^n))^2\langle\nabla\rho(Y_s^n), Z_s^n\rangle^2 ds\Big)^{1/2}\Big]\\
&+ C\E\E^m\Big[\Big(\Int_0^T (\rho(Y_s^n))^2\langle\nabla\rho(Y_s^n), h_s(Y_s^n,Z_s^n)\rangle^2 ds\Big)^{1/2}\Big]\\
&+ C\E\E^m\Big[\Big(\Int_0^T (\rho(Y_s^n))^2|\langle\nabla\rho(Y_s^n), g_s(Y_s^n,Z_s^n)\rangle|^2 ds\Big)^{1/2}\Big].
\end{split}
\end{align}
From the boundedness of $h$ and the fact that $\nabla\rho^2(x)=4\rho(x)$, we have
\be \label{estimate1}
\begin{split}
\E\E^m\Big[\Big(\Int_0^T (\rho(Y_s^n))^2&\langle\nabla\rho(Y_s^n), h(s,B_s,Y_s^n,Z_s^n)\rangle^2 ds\Big)^{1/2}\Big]\leq C \E\E^m\Big[\Big(\Int_0^T (\rho(Y_s^n))^2\rho(Y_s^n) ds\Big)^{1/2}\Big]\\
&\leq C\E\E^m\Big[\underset{0\leq s\leq T}{\Sup}\big(\varphi(Y_s^n)\big) ^{1/2}\Big(\Int_0^T \rho(Y_s^n) ds\Big)^{1/2}\Big]\\
&\leq  \Frac{1}{4}\E\E^m\Big[\underset{0\leq s\leq T}{\Sup}\varphi(Y_s^n)\Big]+C^2 \E\E^m\Big[\Int_0^T (d(Y_s^n,D))^2ds\Big].
\end{split}
\ee
Similarly,
\be \label{2.9}
\begin{split}
&\E\E^m\Big[\Big(\Int_0^T (\rho(Y_s^n))^2|\langle\nabla\rho(Y_s^n), g(s,B_s,Y_s^n,Z_s^n)\rangle|^2 ds\Big)^{1/2}\Big]\\
&\leq  \Frac{1}{4}\E\E^m\Big[\underset{0\leq s\leq T}{\Sup}\varphi(Y_s^n)\Big]+C^2 \E\E^m\Big[\Int_0^T (d(Y_s^n,D))^2ds\Big].
\end{split}
\ee
By the Holder's inequality, we obtain
\be \label{estimate2}
\begin{split}
\E\E^m\Big[\Big(\Int_0^T (\rho(Y_s^n))^2\langle\nabla\rho(Y_s^n), Z_s^n\rangle^2 ds\Big)^{1/2}\Big]&\leq C \E\E^m\Big[\underset{0\leq s\leq T}{\Sup}\big(\varphi(Y_s^n)\big) ^{1/2}\Big(\Int_0^T\langle\nabla\rho(Y_s^n), Z_s^n\rangle^2 ds\Big)^{1/2}\Big]\\
&\leq  \Frac{1}{4}\E\E^m\Big[\underset{0\leq s\leq T}{\Sup}\varphi(Y_s^n)\Big]+C^2 \E\E^m\Big[\Int_0^T\langle\nabla\rho(Y_s^n), Z_s^n\rangle^2 ds\Big].
\end{split}
\ee
Substituting \eqref{estimate1}, (\ref{2.9}) and \eqref{estimate2} in \eqref{uniformestimate} leads to
\be \label{estimate3}
\begin{split}
\E\E^m[\underset{0\leq t\leq T}{\Sup}\varphi(Y_t^n)]&\leq C \E\E^m[\Int_0^T\varphi(Y_s^n) ds]+C\E\E^m\Big[\Int_0^T (d(Y_s^n,D))^{2}ds\Big]\\
&+C^2 \E\E^m\Big[\Int_0^T\langle\nabla\rho(Y_s^n), Z_s^n\rangle^2 ds\Big].
\end{split}
\ee
Since $ \rho$ is a convex function,  Hessian of $ \rho(Y_s^n)$ is a positive semidefinite matrix.  So we have
$ \E\big[\Int_t^T trace [Z_s^nZ_s^{n *} \rho(Y_s^n)Hess \rho(Y_s^n)]ds\big] \geq 0$. By the equation \eqref{hessphi}, we can deduce that
\be
2\E\Big[\Int_0^T\langle\nabla\rho(Y_s^n), Z_s^n\rangle^2 ds\Big]\leq \E\Big[\Int_0^T trace[Z_s^nZ_s^{n *} Hess \varphi(Y_s^n)]ds\Big],
\ee   and from \eqref{estZ}, we get

\b*
\E\E^m\Big[\Int_0^T\langle\nabla\rho(Y_s^n), Z_s^n\rangle^2 ds\Big]\longrightarrow 0 \, \text{as}\, n\rightarrow\infty.
\e*
Finally, by using \eqref{unifY}, \eqref{estimate3} and Lemma \ref{lem}, we get the desired result.
\ep
\end{proof}
\vspace{0.2cm}
\begin{Lemma} \label{lemma-conv}
The sequence $(Y^n,Z^n)$ is a Cauchy sequence in ${\mathcal S}^2_{k}([0,T]) \times {\mathcal H}^2_{k\times d}([0,T])$, i.e.
\begin{equation}
\label{conv}
\E\E^m[\underset{0\leq t\leq T}{\Sup}|Y_t^n- Y_t^p|^2 + \Int_0^T \|Z_t^n-Z_t^p\|^2dt]\longrightarrow 0 \quad\text{as}\quad n,p \rightarrow +\infty.
\end{equation}
\end{Lemma}
\begin{proof}
For all $n,p\geq 0$, we apply It\^o formula to $|Y_t^n-Y_t^p|^2$
\be \label{itoexis}
\begin{split}
|Y_t^n-Y_t^p|^2&+\Int_t^T \|Z_s^n-Z_s^p\|^2ds  =  2\Int_t^T (Y_s^n-Y_s^p)^\top(f(s,B_s,Y_s^n,Z_s^n)-f(s,B_s,Y_s^p,Z_s^p))ds\\
& +2\Int_t^T (Y_s^n-Y_s^p)^\top(h(s,B_s,Y_s^n,Z_s^n)-h(s,B_s,Y_s^p,Z_s^p))d\W_s - 2\Int_t^T (Y_s^n-Y_s^p)(Z_s^n-Z_s^p)dB_s\\
&+\Int_t^T(Y_s^n-Y_s^p)^\top(g(s,B_s,Y_s^n,Z_s^n)-g(s,B_s,Y_s^p,Z_s^p))\ast dB_s\\
&-2\Int_t^Ttrace[(Z_s^n-Z_s^p)^\top(g(s,B_s,Y_s^n,Z_s^n)-g(s,B_s,Y_s^p,Z_s^p))]ds\\
& +\Int_t^T \|h(s,B_s,Y_s^n,Z_s^n)-h(s,B_s,Y_s^p,Z_s^p)\|^2ds- 2n\Int_t^T (Y_s^n-Y_s^p)^\top(Y_s^n-\pi(Y_s^n))ds\\
&+2p\Int_t^T (Y_s^n-Y_s^p)^\top(Y_s^p-\pi(Y_s^p))ds.
\end{split}
\ee
Taking into account the assumptions on $g$ and the inequality $2ab\leq \epsilon a^2+ \epsilon^{-1} b^2$, for all $\epsilon >0$, we have
\begin{align}\label{2.10}
\begin{split}
-2\Int_t^T trace[(Z_s^n-Z_s^p)^\top(g(s,B_s,Y_s^n,Z_s^n)&-g(s,B_s,Y_s^p,Z_s^p))]ds\\
& \leq 2\Int_t^T \|Z_s^n-Z_s^p\|\Big(C|Y_s^n-Y_s^p| + \alpha\|Z_s^n-Z_s^p\|\Big)ds\\
& \leq \epsilon^{-1}\Int_t^T |Y_s^n-Y_s^p|^2ds + (2\alpha+\epsilon)\Int_t^T\|Z_s^n-Z_s^p\|^2ds
\end{split}
\end{align}
By the property (\ref{prop2}), we have
\begin{eqnarray}
\begin{split}
- 2n\Int_t^T (Y_s^n-Y_s^p)^\top(Y_s^n-\pi(Y_s^n))ds &+ 2p\Int_t^T (Y_s^n-Y_s^p)^\top(Y_s^p-\pi(Y_s^p))ds\\
&\leq 2(n+p)\Int_t^T (Y_s^n-\pi(Y_s^n))^\top(Y_s^p-\pi(Y_s^p))ds.\\
&
\end{split}
\end{eqnarray}
Hence, from the Lipschitz continuity on $f$ and $h$, and taking expectation yields
\begin{eqnarray}\label{itoestimate}
\begin{split}
\E\E^m[|Y_t^n-Y_t^p|^2] &+ \E\E^m[\Int_t^T \|Z_s^n-Z_s^p\|^2ds]\leq  2\E\E^m[\Int_t^T C(|Y_s^n-Y_s^p|^2+ |Y_s^n-Y_s^p|\|Z_s^n-Z_s^p\|)ds]\\
&+\E\E^m[\Int_t^T \big((C+\epsilon^{-1})|Y_s^n-Y_s^p|^2+ (2\alpha+\epsilon+\beta^2)\|Z_s^n-Z_s^p\|^2\big)ds]\\
&+ 2(n+p)\E\E^m[\Int_t^T (Y_s^n-\pi(Y_s^n))^\top(Y_s^p-\pi(Y_s^p))ds].
\end{split}
\end{eqnarray}
For the last term, we need the following lemma whose proof is postponed to the Appendix.
\begin{Lemma}\label{extraestimate}
There exists a constant $C>0$ such that, for each $n\geq 0$,
\be
 \E\E^m\big[\Big(n\Int_0^T d(Y_s^n,D) ds\Big)^2\big]\leq C.
 \label{boundvariation}
\ee
\end{Lemma}
\vspace{0.5cm}
\noindent Now we can deduce from the H\"older inequality and Lemma \ref{extraestimate} that
\be \label{estdistance}
\begin{split}
n\E\E^m[\Int_t^T &(Y_s^n-\pi(Y_s^n))^\top(Y_s^p-\pi(Y_s^p))ds] \leq  n\E\E^m[\Int_t^T  d(Y_s^n,D)d(Y_s^p,D))ds]\\
&\leq n\E\E^m[\underset{0\leq s\leq T}{\Sup}d(Y_s^p,D)\Int_t^T  d(Y_s^n,D)ds)]\\
 &\leq \Big(\E\E^m\big[\Big(n\Int_0^T d(Y_s^n,D) ds\Big)^2\big]\Big)^{1/2} \Big(\E\E^m[\underset{0\leq s\leq T}{\Sup}d^2(Y_s^p,D)]\Big)^{1/2}\\
&\leq C\Big(\E\E^m[\underset{0\leq s\leq T}{\Sup}d^2(Y_s^p,D)]\Big)^{1/2}.
\end{split}
\ee
Substituting (\ref{estdistance}) in the previous inequality \eqref{itoestimate}, we have
\b*
\begin{split}
\E\E^m[|Y_t^n-Y_t^p|^2] &+ (1-\beta^2-2 \alpha-C\epsilon)\E\E^m[\Int_t^T \|Z_s^n-Z_s^p\|^2ds]\leq  C(1+\epsilon^{-1})\E\E^m[\Int_t^T |Y_s^n-Y_s^p|^2ds]\\
&+ C\Big(\E\E^m[\underset{0\leq s\leq T}{\Sup}d^2(Y_s^n,D)]\Big)^{1/2}+ C\Big(\E\E^m[\underset{0\leq s\leq T}{\Sup}d^2(Y_s^p,D)]\Big)^{1/2}.
\end{split}
\e*
Choosing $1-\beta^2-2 \alpha-C\epsilon>0$, by Gronwall's lemma, we obtain
\be
\underset{0\leq t\leq T}{\Sup}\E\E^m[|Y_t^n-Y_t^p|^2]\leq C\Big(\E\E^m[\underset{0\leq s\leq T}{\Sup}d^2(Y_s^p,D)]\Big)^{1/2}+ C\Big(\E\E^m[\underset{0\leq s\leq T}{\Sup}d^2(Y_s^n,D)]\Big)^{1/2}.
\ee
We deduce similarly
\be
 \E\E^m[\Int_0^T \|Z_s^n-Z_s^p\|^2ds]\leq C\Big(\E\E^m[\underset{0\leq s\leq T}{\Sup}d^2(Y_s^p,D)]\Big)^{1/2}+ C\Big(\E\E^m[\underset{0\leq s\leq T}{\Sup}d^2(Y_s^n,D)]\Big)^{1/2}.
\label{estunifZ}
\ee
Next, by \eqref{itoexis}, the Burkholder-Davis-Gundy inequality and the previous calculations we get
\begin{align*}
&\E\E^m[\underset{0\leq t\leq T}{\Sup}|Y_t^n-Y_t^p|^2]\leq C\E\E^m[\Int_0^T | Y_s^n-Y_s^p||f(s,Y_s^n,Z_s^n)-f(s,Y_s^p,Z_s^p)|ds]\\
&+C\E\E^m\big(\Int_0^T | Y_s^n-Y_s^p|^2\|h(s,B_s,Y_s^n,Z_s^n)-h(s,B_s,Y_s^p,Z_s^p)\|^2ds\big)^{1/2}+C\E\E^m\big(\Int_0^T | Y_s^n-Y_s^p|^2\| Z_s^n-Z_s^p\|^2ds\big)^{1/2}\\
&+C\E\E^m\big(\Int_0^T | Y_s^n-Y_s^p|^2\|g(s,B_s,Y_s^n,Z_s^n)-g(s,B_s,Y_s^p,Z_s^p)\|^2ds\big)^{1/2}\\
&+\E\E^m[\Int_0^T C(|Y_s^n-Y_s^p|^2+ \alpha\|Z_s^n-Z_s^p\|^2)ds]+  2(n+p)\E[\Int_0^T (Y_s^n-\pi(Y_s^n))^\top(Y_s^p-\pi(Y_s^p))ds].
\end{align*}
Then, it follows by the Lipschitz Assumption \ref{assgener} on $f$, $g$ and $h$ and  (\ref{estdistance}) that for any $n,p\geq 0$
\b*
\begin{split}
\E\E^m[\underset{0\leq t\leq T}{\Sup}|Y_t^n-Y_t^p|^2]&\leq  C\Big(\E\E^m[\underset{0\leq s\leq T}{\Sup}d^2(Y_s^p,D)]\Big)^{1/2}+C\Big(\E\E^m[\underset{0\leq s\leq T}{\Sup}d^2(Y_s^n,D)]\Big)^{1/2}\\
&+ C\E\E^m(\underset{0\leq t\leq T}{\Sup}|Y_t^n-Y_t^p|^2\Int_0^T \|Z_s^n-Z_s^p\|^2ds)^{1/2}\\
&\leq  C\Big(\E\E^m[\underset{0\leq s\leq T}{\Sup}d^2(Y_s^p,D)]\Big)^{1/2}+C\Big(\E\E^m[\underset{0\leq s\leq T}{\Sup}d^2(Y_s^n,D)]\Big)^{1/2}\\
&+ C\varepsilon\E\E^m(\underset{0\leq t\leq T}{\Sup}|Y_t^n-Y_t^p|^2)+ C\varepsilon^{-1}\E\E^m(\Int_0^T \|Z_s^n-Z_s^p\|^2ds).
\end{split}
\e*
Choosing $1-C\varepsilon>0$ and from the inequality (\ref{estunifZ}) we conclude that
\b*
\begin{split}
\E\E^m[\underset{0\leq t\leq T}{\Sup}|Y_t^n-Y_t^p|^2] &\leq C\Big(\E\E^m[\underset{0\leq s\leq T}{\Sup}d^2(Y_s^p,D)]\Big)^{1/2}+C\Big(\E\E^m[\underset{0\leq s\leq T}{\Sup}d^2(Y_s^n,D)]\Big)^{1/2}\\
&\leq C\Big(\E\E^m[\underset{0\leq s\leq T}{\Sup}d^4(Y_s^p,D)]\Big)^{1/4}+C\Big(\E\E^m[\underset{0\leq s\leq T}{\Sup}d^4(Y_s^n,D)]\Big)^{1/4}\longrightarrow 0,
\end{split}
\e*
as $n,m\rightarrow \infty$, where Lemma \ref{fundamental:lemma} has been used.\ep
\end{proof}
\vspace{0.6cm}
Consequently, since for any $n,p\geq 0$ and $0\leq t\leq T$,
\be
\begin{split}
K_s^{n}-K_s^{p} &= Y_0^{n}-Y_0^{p}-Y_s^{n}-Y_s^{p}-\Int_0^s (f(r,Y_r^{n},Z_r^{n})-f(r,Y_r^{p},Z_r^{p}))dr\\
&-\Int_0^s (h(r,Y_r^{n},Z_r^{n})-h(r,Y_r^{p},Z_r^{p}))d\W_r
 +\Int_0^s (Z_r^{n},-Z_r^{p}) dB_r.
\end{split}
\ee
we obtain from \eqref{conv} and Burkholder-Davis-Gundy inequality,
\be\label{convK}
\E\E^m(\underset{0\leq s\leq T}{\Sup} |K_s^{n}-K_s^{p}|^2)\rightarrow 0 \quad \mbox{as}~~ n, p\rightarrow\infty.
\ee
We have also the following result:
\begin{Lemma} There exists a $\Fc_s$ measurable triple processes $(Y_s,Z_s,K_s)_{s\in[0,T]}$ such that
\be
\E\E^m\big(\underset{0\leq s\leq T}{\Sup} |Y_s^{n}-Y_s|^2 + \Int_0^T |Z_s^{n}-Z_s|^2 ds+\underset{0\leq s\leq T}{\Sup} |K_s^{n}-K_s|^2\big) \rightarrow 0 \quad \mbox{as}~~ n\rightarrow\infty.\nonumber\\
\quad
\label{converg}
\ee
Moreover, this triple of processes is a solution of the following RBDSDE:
\begin{equation}
\label{RBDSDE}
 \left\lbrace
\begin{aligned}
&(i)~ Y_{s}
 =\Phi(B_{T})+
\Int_{s}^{T}f(r,B_{r},Y_{r},Z_{r})dr+\Int_{s}^{T}h(r,B_{r},Y_{r},Z_{r})d\W_r+K_{T}-K_{s}\\
&\hspace{2.5cm}
 +\frac{1}{2}\Int_t^T g(r,B_{r},Y_{r},Z_{r})\ast dB_r-\Int_{s}^{T}Z_{r}dB_{r},\;
\P\otimes \P^m\text{-}a.s. , \; \forall \,  s \in [t,T]  \\
& (ii)~ Y_{s} \in \bar{D} \, \, \quad \P\otimes \P^m\text{-}a.s.\\
& (iii) \Int_0^T (Y_{s}-v_s(B_{s}))^\top dK_{s}\leq 0., \,\P\otimes \P^m\text{-}a.s., \\
&~ \text{for any continuous }\, \Fc_s -\text{random function}\, v \, : \,[0,T] \times \Omega \times \mathbb R^d \longrightarrow \,  \bar{D}.
\end{aligned}
\right.
\end{equation}
\end{Lemma}
\begin{proof}
First, we have from \eqref{conv} that $(Y^{n},Z^{n})$ is a Cauchy sequence in ${\Sc}^2_k([0,T])\times {\Hc}^2_{k\times d}([0,T])$ and therefore there exists a unique pair $(Y_s,Z_s)$  of $\Fc_s$- measurable processes which valued in $\R^k\times\R^{k\times d}$, satisfying
\be
\E\E^m(\underset{0\leq s\leq T}{\Sup} |Y_s^{n}-Y_s|^2 + \Int_0^T |Z_s^{n}-Z_s|^2 ds) \rightarrow 0 \quad \mbox{as}~~ n\rightarrow\infty.
\label{ConvYZ}
\ee
Similarly, we obtain from \eqref{convK} there exists a $\Fc_s$- adapted continuous process $(K_s)_{0\leq s\leq T}$ ( with $K_0=0$) such that $$\E\E^m(\underset{0\leq s\leq T}{\Sup} |K_s-K_s^{n}|^2 )\rightarrow 0 \quad \mbox{as}~~ n\rightarrow\infty.$$
Furthermore, \eqref{estunif} shows that the total variation of $K^{n}$ is uniformly bounded. Thus, $K$ is also of uniformly bounded variation.
Passing to the limit in \eqref{BDSDEpen}, the processes $(Y_s,Z_s,K_s)_{0\leq s\leq T}$ satisfy
\be
Y_{s}
 &=&\Phi(B_{T})+
\Int_{s}^{T}f(r,B_{r},Y_{r},Z_{r})dr+\Int_{s}^{T}h(r,B_{r},Y_{r},Z_{r})d\W_r+K_{T}-K_{s}\nonumber\\
 &+&\frac{1}{2}\Int_t^T g(r,B_{r},Y_{r},Z_{r})\ast dB_r-\Int_{s}^{T}Z_{r}dB_{r},\;
\P\otimes \P^m\text{-}a.s. , \; \forall \,  s \in [t,T]
\ee
Since we have from Lemma \ref{fundamental:lemma} that $Y_s$ is in $\bar{D}$, it remains to check the minimality property for $ (K_s)$, namely i.e., for any continuous $\Fc_s $- random function $ v$ valued in $\bar{D}$, $$\Int_0^T (Y_s-v_s(B_{s}))^\top dK_s \leq 0.$$
We note that \eqref{prop1} gives us
$$\Int_0^T (Y_s^{n}-v_s(B_{s}))^\top dK_s^{n}= -n\Int_0^T (Y_s^{n}-v_s(B_{s})^\top (Y_s^{n}-\pi(Y_s^{n}))ds \leq 0.$$
Therefore, we will show that we can extract a subsequence such that $\Int_0^T (Y_s^{n}-v_s(B_{s}))^\top dK_s^{n}$ converge a.s. to $\Int_0^T (Y_s-v_s(B_{s}))^\top dK_s.$ Following the proof of Lemma \ref{estapriori} in Appendix, we have
\begin{align*}
2\gamma \|K^{n}\|_{VT}&\leq |\Phi(B_{T})|^2 + 2\Int_0^T (Y_s^n)^\top f(s,B_s,Y_s^n,Z_s^n)ds +2\Int_0^T (Y_s^n)^\top h(s,B_s,Y_s^n,Z_s^n)d\W_s \\
&- 2\Int_0^T (Y_s^n)Z_s^ndB_s+\Int_0^T(Y_s^n)^\top g(s,B_s,Y_s^n,Z_s^n)\ast dB_s\\
&-2\Int_0^T trace[(Z_s^n)^\top g(s,B_s,Y_s^n,Z_s^n)]ds +\Int_0^T \|h(s,B_s,Y_s^n,Z_s^n)\|^2ds.
\end{align*}
Notice that the right hand side tends in probability as $n$ goes to infinity  to
\begin{align*}
&|\Phi(B_{T})|^2 + 2\Int_0^T (Y_s)^\top f(s,B_s,Y_s,Z_s)ds +2\Int_0^T (Y_s)^\top h(s,B_s,Y_s,Z_s)d\W_s
- 2\Int_0^T (Y_s)Z_sdB_s\\
&+\Int_0^T(Y_s)^\top g(s,B_s,Y_s,Z_s)\ast dB_s
-2\Int_0^T trace[(Z_s)^\top g(s,B_s,Y_s,Z_s)]ds +\Int_0^T \|h(s,B_s,Y_s,Z_s)\|^2ds.
\end{align*}
Thus, there exists a subsequence $(\phi(n))_{n\geq 0}$ such that the convergence is almost surely and $\|K^{\phi(n)}\|_{VT}$ is bounded. Moreover, due to the convergence in $\L^2$ of $\underset{0\leq s\leq T}{\Sup} |Y_s^{n}-Y_s|^2$ to $0$, we can extract a subsequence  from $(\phi(n))_{n\geq 0}$ such that $Y^{\phi(\psi(n))}$ converges uniformly to $Y$. Hence, we apply  Lemma 5.8 in \cite{GP} and we obtain
$$\Int_0^T(Y_s^{\phi(\psi(n))}-v_s(B_{s}))^\top dK_s^{\phi(\psi(n)))}\longrightarrow\Int_0^T(Y_s -v_s(B_{s}))^\top dK_s\quad a.s. ~~\mbox{as}~~ n\rightarrow\infty$$
which is  the required result.
\ep
\end{proof}
\vspace{0.5cm}
\noindent We remind that the purpose of this section is to prove that the penalized solution $(u^n)_n$ is a Cauchy sequence. By all the calculations done before we obtain:
\begin{eqnarray*}
&&\E[\underset{0\leq t \leq T}{\Sup}\|u^{n}(t)-u^{p}(t)\|_2 ^{2}+\int_0^T\|u^{n}(t)-u^{p}(t)\|^2dt]\\
&=&\E[\underset{0\leq t \leq T}{\Sup} \Int_{\mathbb{R}^{d}}\left|
u^{n}(t,x)-u^{p}(t,x)\right| ^{2}dx+\int_0^T\Int_{\mathbb{R}^{d}}\left| (\nabla u^{n})(s,x)-(\nabla u^{p})(s,x)\right| ^{2} dsdx] \\
&=&\E[\underset{0\leq t \leq T}{\Sup} \E^m[\left|
Y^{n}(t)-Y^{p}(t)\right| ^{2}]+\E^m[\int_0^T\|Z^{n}_s)-Z^{p}_s\| ^{2} ds]] \\
&\leq &\E\E^m[\underset{0\leq t \leq T}{\Sup}\left|
Y_{s}^{n}-Y_{s}^{p}\right| ^{2}+\Int_{0}^{T}\left\|Z_{s}^{n}-Z_{s}^{p}\right\| ^{2}ds] \longrightarrow 0.
\end{eqnarray*}
Therefore $(u_{n})_{n\in\mathbb N}$ is a Cauchy sequence in $\mathcal{H}_T$, and the limit $%
u=\underset{n\rightarrow \infty }{\lim}u_{n}$ belongs to  $\mathcal{H}_T$.\\

Denote $\nu _{n}(dt,dx)= -n(u_{n}-\pi(u_{n}))(t,x)dtdx$. Then by (\ref{estimateL4-1})  we have
\begin{eqnarray}\label{totalvariation}
&&\underset{n}{\Sup}\,\E\big[Var(\nu _{n})([0,T]\times \mathbb{R}^{d})^2\big] \nonumber\\
&=&\underset{n}{\Sup}\E[\left (n\Int_{0}^{T}\Int_{\R^d}|(u_{n}-\pi(u_{n}))(s,x)|dsdx\right )^2]<\infty ,
\end{eqnarray}
where $Var(\nu_{n})$ denotes the total variation of $\nu_n$ on $Q_T=[0,T]\times \R^d$.  Let ${\cal M}({Q}_T)$ denote the Banach space of totally finite signed measures on $Q_T$ ($\R^k$-valued), equipped with the norm of total variation. (\ref{totalvariation}) implies that $\{\nu_n(dt,dx), n\geq 1\}$ is bounded in $L^2(\Omega, {\cal M}({Q}_T))$, hence  relatively compact with respect to the weak$^*$ topology in
  $L^2(\Omega, {\cal M}({Q}_T))$. Thus, we may assume ( take a subsequence if necessary) that $\nu_n$  converges to some random signed measure  $\nu\in L^2(\Omega, {\cal M}({Q}_T))$ with respect to the weak$^*$ topology. Moreover, it follows from (\ref{totalvariation}) that $\E[(Var(\nu)([0,T]\times \R^d))^2]<\infty$ for every $T>0$.
Hence, for $\varphi \in \mathcal{D}_T$ with compact support in $x$,
$$
\Int_{\mathbb{R}^{d}}\Int_{t}^{T}\varphi d\nu _{n} \rightarrow \Int_{\mathbb{R}^{d}}\Int_{t}^{T}\varphi d\nu
$$
weakly in $L^2(\Omega)$.
Now passing to the limit in the SPDE $(\Phi,f^{n},g,h)$ (\ref{o-equa1}), we get that $(u,\nu )$
satisfies the reflected SPDE associated to $(\Phi,f,g,h)$, i.e. for every $\varphi \in
\mathcal{D}_T$, we have
\begin{align} \label{equa1}
 \nonumber\Int_{t}^{T}&\big[(u(s,\cdot),\partial
_{s}\varphi(s,\cdot) )+\Frac{1}{2}(\nabla u(s,\cdot),\nabla\varphi(s,\cdot))\big]ds +(u(t,\cdot ),\varphi
(t,\cdot ))-(\Phi(\cdot ),\varphi (T,\cdot))\\
\nonumber &=\Int_{t}^{T}\big[(f(s,\cdot,u^{n}(s,\cdot),\nabla u(s,\cdot)),\varphi(s,\cdot))+(g(s,\cdot,u(s,\cdot),\nabla u(s,\cdot)),\varphi(s,\cdot))\big]ds\\
&+\Int_{t}^{T}(h(s,\cdot,u(s,\cdot),\nabla u(s,\cdot)),\varphi(s,\cdot))d\W_s+\Int_{t}^{T}\Int_{\mathbb{R}^{d}}\varphi (s,x)\nu (ds,dx).
\end{align}

\vspace{0.3cm}
\noindent We can now deduce following the probabilistic
interpretation (Feymamn-Kac's formula) for the measure $\nu $ via
the nondecreasing process $K^{t,x}$ of the RBDSDE \eqref{RBDSDE}.
\begin{Lemma}
We have
$$u(t,x)\in \bar{D}, \quad dx\otimes dt\otimes d\P-a.e., \quad\quad \nu
(ds,dx)=1_{\{u\in\partial D\}}(s,x)\nu (ds,dx). $$
Furthermore, for every measurable bounded and positive functions $\varphi $ and $\psi $,
\begin{align}
\Int_{\mathbb{R}^{d}}\Int_{t}^{T}\varphi (s,B^{-1}_{s})\psi (s,x)\nu^i (ds,dx)=\Int_{\mathbb{R}^{d}}\Int_{t}^{T}\varphi (s,x)\psi (s,B_{s})dK_{s}^{i}dx\text{, a.s..}
\label{con-k}
\end{align}
\end{Lemma}
\begin{proof} Since $K^{n}$ converges to $K$ uniformly in $t$, the
measure $dK^{n}$ converges to $dK$ weakly in probability.\\
Fix two continuous functions $\varphi $, $\psi $ : $[0,T]\times \mathbb{R}%
^{d}\rightarrow \mathbb{R}^{+}$ which have compact support in $x$ and a
continuous function with compact support $\theta :\mathbb{R}^{d}\rightarrow %
\mathbb{R}^{+}$, from Bally et al \cite{BCKF} (The proof of Theorem 4), we have (see also Matoussi and Xu \cite{MX08})
\begin{eqnarray*}
&&\Int_{\mathbb{R}^{d}}\Int_{t}^{T}\varphi (s,B^{-1}_{t,s}(x))\psi (s,x)\theta (x)\nu (ds,dx) \\
&=&\lim_{n\rightarrow \infty }-\Int_{\mathbb{R}^{d}}\Int_{t}^{T}\varphi (s,
B^{-1}_{t,s}(x))\psi (s,x)\theta(x)n(u_{n}-\pi(u_{n}))(s,x)dsdx \\
&=&\lim_{n\rightarrow \infty }-\Int_{\mathbb{R}^{d}}\Int_{t}^{T}\varphi(s,x)\psi (s,B_{t,s}(x))\theta
(B_{t,s}(x))n(u_{n}-\pi(u_{n}))(t,B_{t,s}(x))dtdx \\
&=&\lim_{n\rightarrow \infty }\Int_{\mathbb{R}^{d}}\Int_{t}^{T}\varphi(s,x)\psi (s,B_{t,s}(x))\theta (B_{t,s}(x))dK_{s}^{n,t,x}dx \\
&=&\Int_{\mathbb{R}^{d}}\Int_{t}^{T}\varphi (s,x)\psi (s,B_{t,s}(x))\theta
(B_{t,s}(x))dK_{s}^{t,x}dx.
\end{eqnarray*}
We take $\theta =\theta _{R}$ to be the regularization of the indicator
function of the ball of radius $R$ and pass to the limit with $R\rightarrow
\infty $, to get that
\begin{equation}\label{con-k1}
\Int_{\mathbb{R}^{d}}\Int_{t}^{T}\varphi (s,B^{-1}_{t,s}(x))\psi (s,x)\nu (ds,dx)=\Int_{\mathbb{R}^{d}}\Int_{t}^{T}\varphi
(s,x)\psi (s,B_{t,s}(x))dK_{s}^{t,x}dx.
\end{equation}
We know that $dK_{s}^{t,x}=1_{\{Y^{t,x}_s\in\partial D\}}dK_{s}^{t,x}=1_{\{u\in\partial D\}}(s,B_{t,s}(x))dK_{s}^{t,x}$. Again by regularization procedure we can set $\psi =1_{\{u\in\partial D\}}$  in \eqref{con-k1} to obtain
\begin{align*}
\Int_{\mathbb{R}^{d}}\Int_{t}^{T}\varphi (s,B^{-1}_{t,s}(x))1_{\{u\in\partial D\}}(s,x)\nu (ds,dx)
=\Int_{\mathbb{R}^{d}}\Int_{t}^{T}\varphi(s,B^{-1}_{t,s}(x))\nu (ds,dx)\text{, a.s.}
\end{align*}
Note that the family of functions $A(\omega )=\{(s,x)\rightarrow \phi (s,
B^{-1}_{t,s}(x)):\varphi \in C_{c}^{\infty }\}$ is an algebra which
separates the points (because $x\rightarrow B^{-1}_{t,s}(x)$ is a
bijection). Given a compact set $G$, $A(\omega )$ is dense in $C([0,T]\times
G)$. It follows that $J(B^{-1}_{t,s}(x))1_{\{u\in\partial D\}}(s,x)\nu (ds,dx)=J(
B^{-1}_{t,s}(x))\nu (ds,dx)$ for almost every $\omega $. While $J(B^{-1}_{t,s}(x))>0$ for almost every $\omega $, we get $\nu
(ds,dx)=1_{\{u\in\partial D\}}(s,x)\nu (ds,dx)$, and \eqref{con-k} follows.\\
Then we get easily that $Y_{s}^{t,x}=u(s,B_{t,s}(x))$ and $Z_{s}^{t,x}=(\nabla u\sigma)(s,B_{t,s}(x))$,  in view of the convergence
results for $(Y_{s}^{n,t,x},Z_{s}^{n,t,x})$ and the flow property associated to $B$. So $u(s,B_{t,s}(x))=Y_{s}^{t,x}\in \bar{D}$. Specially for $s=t$, we
have $u(t,x)\in \bar{D}$.\ep
\end{proof}
\subsection{Proof of uniqueness}
Let $(u^1, \nu^1)$, $(u^2, \nu^2)$ be two solutions of the reflected SPDEs.  Denote by  $(Y^1,Z^1,K^1)$ and $(Y^2,Z^2,K^2)$ the associated  solutions of  the RBDSDE  \eqref{RBDSDE}. To show the uniqueness it suffices to prove $(Y^1,Z^1,K^1)=(Y^2,Z^2,K^2)$.  Applying the double stochastic It\^o's formula extended in Matoussi and Stoica (Corollay 1 and Remark 2 in \cite{MS10} p.1158) yields
\be \label{itouni}
\begin{split}
|Y_t^1-Y_t^2|^2&+\Int_t^T \|Z_s^1-Z_s^1\|^2ds  =  2\Int_t^T (Y_s^1-Y_s^2)^\top(f(s,B_s,Y_s^1,Z_s^1)-f(s,B_s,Y_s^2,Z_s^2))ds\\
& +2\Int_t^T (Y_s^1-Y_s^2)^\top(h(s,B_s,Y_s^1,Z_s^1)-h(s,B_s,Y_s^2,Z_s^2))d\W_s - 2\Int_t^T (Y_s^1-Y_s^2)(Z_s^1-Z_s^2)dB_s\\
&+\Int_t^T(Y_s^1-Y_s^2)^\top(g(s,B_s,Y_s^1,Z_s^1)-g(s,B_s,Y_s^2,Z_s^2))\ast dB_s\\
&-2\Int_t^Ttrace[(Z_s^1-Z_s^2)^\top(g(s,B_s,Y_s^1,Z_s^1)-g(s,B_s,Y_s^2,Z_s^2))]ds\\
& +\Int_t^T \|h(s,B_s,Y_s^1,Z_s^1)-h(s,B_s,Y_s^2,Z_s^2)\|^2ds+2\Int_t^T (Y_s^1-Y_s^2)^\top(dK_s^1-dK_s^2).
\end{split}
\ee
Therefore, under the minimality condition (iv) we have
\be\label{minest}
\Int_t^T (Y_s^1-Y_s^2)^\top(dK_s^1-dK_s^2)\leq 0 ,\quad \text{for all}~~ t\in[0,T].
\ee
Then, following the proof of Lemma \ref{lemma-conv} we obtain
\b*
\begin{split}
\E\E^m[|Y_t^1-Y_t^2|^2] &+ (1-\beta^2-2 \alpha-C\epsilon)\E\E^m[\Int_t^T \|Z_s^1-Z_s^2\|^2ds]\leq  C(1+\epsilon^{-1})\E\E^m[\Int_t^T |Y_s^1-Y_s^2|^2ds]
\end{split}
\e*
Choosing $1-\beta^2-2 \alpha-C\epsilon>0$ and from Gronwall's lemma,
$$\E\E^m[|Y_t^1-Y_t^2|^2] = 0 \,\,,\quad \E\E^m[\Int_0^T \|Z_s^1-Z_s^2\|^2ds] =0,\,\,\,\, 0\leq t\leq T.$$

\appendix

\section{A priori estimates}
\label{A priori estimates}
In this section, we provide a priori estimates which are uniform in $n$ on the solutions of (\ref{BDSDEpen}).
\begin{Lemma}\label{estapriori}
There exists a constant $C>0$, independent of $n$, such that for all $n$ large enough
\be
\underset{n}{\Sup}\,\E\E^m\big[\underset{0\leq t\leq T}{\Sup}|Y_t^n|^2+\Int_t^T \|Z_s^n\|^2 ds + \|K^n\|_{VT}\big]\leq C\Big[\|\Phi(B_{T})\|_{2}^2+\Int_t^T \big(\|f_s^0\|_{2,2}^2+\|h_s^0\|_{2,2}^2+\|g_s^0\|_{2,2}^2\big)ds\Big].\nonumber\\
\quad
\label{estunif}
\ee
\end{Lemma}
\begin{proof}
We apply generalized It\^o's formula to get
\be \label{estimationuniforme}
\begin{split}
|Y_t^n|^2&+\Int_t^T \|Z_s^n\|^2ds  = |\Phi(B_{T})|^2 + 2\Int_t^T (Y_s^n)^\top f(s,B_s,Y_s^n,Z_s^n)ds +2\Int_t^T (Y_s^n)^\top h(s,B_s,Y_s^n,Z_s^n)d\W_s \\
& - 2\Int_t^T (Y_s^n)Z_s^ndB_s +\Int_t^T(Y_s^n)^\top g(s,B_s,Y_s^n,Z_s^n)\ast dB_s-2\Int_t^T trace[(Z_s^n)^\top g(s,B_s,Y_s^n,Z_s^n)]ds\\
& +\Int_t^T \|h(s,B_s,Y_s^n,Z_s^n)\|^2ds- 2n\Int_t^T (Y_s^n)^\top (Y_s^n-\pi(Y_s^n))ds.
\end{split}
\ee
The stochastic integrals have both zero expectations under $\P\otimes\P^m$ since $(Y^n,Z^n)$ belongs to ${\mathcal S}^2_k([0,T])\times{\mathcal H}^2_{k\times d}([0,T])$. We take expectation in (\ref{estimationuniforme}) and we use conditions (\ref{prop1}) and the Lipschitz Assumption \ref{assgener} in order to obtain
\begin{align}\label{estuniform}
\begin{split}
\E\E^m[|Y_t^n|^2]&+\E\E^m[\Int_t^T \|Z_s^n\|^2ds]\leq\E\E^m[|\Phi(B_{T})|^2]+2C\E\E^m[\Int_t^T |Y_s^n||f^0(s,B_s)|ds]\\
&+2\E\E^m\big[\Int_t^T trace[(Z_s^n)^\top g(s,B_s,Y_s^n,Z_s^n)]ds\big]+\E\E^m[\Int_t^T\|h^0(s,B_s)\|^2ds] .
\end{split}
\end{align}
Taking into account the assumptions on $g$ and the inequality $2ab\leq \epsilon a^2+ \epsilon^{-1} b^2$, for all $\epsilon >0$, we have
\begin{align}\label{esttrace}
\begin{split}
2\Int_t^T trace[&(Z_s^n)^\top(g(s,B_s,Y_s^n,Z_s^n))]ds \leq 2\Int_t^T \|Z_s^n\|\|g^0(s,B_s)\|ds\\
& \leq \epsilon^{-1}\Int_t^T \|g^0(s,B_s)\|^2ds + \epsilon\Int_t^T\|Z_s^n\|^2ds
\end{split}
\end{align}
Plugging estimate \eqref{esttrace} in \eqref{estunif} yields to
\begin{align}\label{estunifor}
\begin{split}
\E\E^m[|Y_t^n|^2]&+\Frac{1}{2}\E\E^m[\Int_t^T \|Z_s^n\|^2ds]
\leq C[\|\Phi(B_{T})\|_{2}^2+\Int_t^T \big(\|f_s^0\|_{2,2}^2+\|h_s^0\|_{2,2}^2+\|g_s^0\|_{2,2}^2\big)ds].
\end{split}
\end{align}
Then, it follows from Gronwall's lemma that
$$\underset{0\leq t\leq T}{\Sup}\E[|Y_t^n|^2]\leq C(\|\Phi(B_{T})\|_{2}^2+\Int_0^T \big(\|f_s^0\|_{2,2}^2+\|h_s^0\|_{2,2}^2+\|g_s^0\|_{2,2}^2\big)ds.$$
Therefore we can deduce

$$\E\E^m[\Int_0^T \|Z_s^n\|^2ds]\leq C(\|\Phi(B_{T})\|_{2}^2+\Int_0^T \big(\|f_s^0\|_{2,2}^2+\|h_s^0\|_{2,2}^2+\|g_s^0\|_{2,2}^2\big)ds.$$

On the other hand, the uniform estimate on $Y^n$ is obtained by taking the supremum over $t$ in the equation (\ref{estimationuniforme}), using the previous calculations and Burkholder-Davis-Gundy inequality. Thus, we get for all $n\geq 0$
\b*
\E\E^m[\underset{0\leq t\leq T}{\Sup}|Y_t^n|^2]\leq C.
\e*
Finally, the total variation of the process $K^n$ is given by
\b*
\|K^n\|_{VT}= n\Int_0^T |Y_s^n-\pi(Y_s^n)|ds.
\e*
But from the property (\ref{prop3}) and the equation (\ref{estimationuniforme}) we have
\begin{eqnarray}\label{variation}
&&2n\Int_t^T \gamma |Y_s^n-\pi(Y_s^n)|ds \leq  2n\Int_t^T |(Y_s^n)^\top (Y_s^n-\pi(Y_s^n))|ds\nonumber\\
&\leq& |\Phi(B_{T})|^2 + 2\Int_t^T (Y_s^n)^\top f(s,B_s,Y_s^n,Z_s^n)ds +2\Int_t^T (Y_s^n)^\top h(s,B_s,Y_s^n,Z_s^n)d\W_s\nonumber \\
&-& 2\Int_t^T (Y_s^n)Z_s^ndB_s +\Int_t^T(Y_s^n)^\top g(s,B_s,Y_s^n,Z_s^n)\ast dB_s\nonumber\\
&-&2\Int_t^T trace[(Z_s^n)^\top g(s,B_s,Y_s^n,Z_s^n)]ds+\Int_t^T \|h(s,B_s,Y_s^n,Z_s^n)\|^2ds
\end{eqnarray}
Hence it follows from previous estimates that
\b*
\E\E^m[ \|K^n\|_{VT}]\leq C(\|\Phi(B_{T})\|_{2}^2+\Int_0^T \big(\|f_s^0\|_{2,2}^2+\|h_s^0\|_{2,2}^2+\|g_s^0\|_{2,2}^2\big)ds,
\e*
and the proof of Lemma \ref{estapriori} is complete.\ep
\end{proof}
\vskip 0.3cm
\noindent Now by taking expectation with respect to $\E^m$ in (\ref{variation}) and then squaring the resulting inequality we can also show that following result.
\begin{Lemma}
\begin{eqnarray} \label{estimateL4-1}
\underset{n}{\Sup}\,\E\big[\left (n\int_0^T\int_{\R^d}|u_n(t,x)-\pi(u_n(t,x))|dtdx\right )^2\big]
  &<& \infty.
\end{eqnarray}
\end{Lemma}
\noindent Next, following the calculations and the estimates done before we can also prove the  $L^4$ estimate for the solutions of (\ref{BDSDEpen}).
\begin{Lemma} There exists a constant $C>0$, independent of $n$, such that for all $n$ large enough
\begin{eqnarray} \label{estimateL4}
\underset{n}{\Sup}\,\E\E^m[\underset{0\leq t\leq t}{\Sup}|Y_t^n|^4+\Big(\Int_0^T\|Z_s^n  \|^2ds\Big)^2]
  &\leq& C \E\E^m\Big[|\Phi(B_{T})|^4+\Int_0^T (|f_s^0|^4+ |h_s^0|^4+ |g_s^0|^4)ds \Big]<\infty.\nonumber\\
  &&
 \quad
\end{eqnarray}
\end{Lemma}
\section{ Proof of Lemma 3.4.}
\label{Proof of Lemma3.4.}
Let first recall that $(Y^n,Z^n)$ is solution of the BDSDE \eqref{BDSDEpen} associated to $(\Phi(B_{T}),f^n,g,h)$ where $f^n(s,y,z)=f(s,y,z)-n(y-\pi(y))$, for each $(y,z)\in\R^k\times\R^{k\times d}$. Note that , since $0\in D$, $f_s^n(0,0)=f_s(0,0):=f_s^0$.
Now, we apply generalized It\^o's formula to get
\begin{align}\label{estimat}
|Y_t^n|^2+\Int_t^T &\|Z_s^n\|^2ds+ 2n\Int_t^T (Y_s^n)^\top(Y_s^n-\pi(Y_s^n))ds= |\Phi(B_{T})|^2 + 2\Int_t^T (Y_s^n)^\top f(s,B_s,Y_s^n,Z_s^n)ds\nonumber\\
&+2\Int_t^T (Y_s^n)^\top h(s,B_s,Y_s^n,Z_s^n)d\W_s
 -2\Int_t^T(Y_s^n)^\top Z_s^n dB_s+\Int_t^T\|h(s,B_s,Y_s^n,Z_s^n)\|^2ds\nonumber\\
 &-\Int_t^T (Y_s^n)^\top g(s,B_s,Y_s^n,Z_s^n)\ast dB_s-2\Int_t^Ttrace[(Z^n_s)^\top g(s,B_s,Y^n_s,Z^n_s)]ds
\end{align}
We apply the property (\ref{prop3}) to obtain
\begin{align*}
\begin{split}
2n\Int_0^T \gamma |Y_s^n-\pi(Y_s^n)|ds &\leq  2n\Int_0^T |(Y_s^n)^\top(Y_s^n-\pi(Y_s^n))|ds\\
&\leq  |\Phi(B_{T})|^2 + 2\Int_0^T (Y_s^n)^\top f(s,B_s,Y_s^n,Z_s^n)ds+2\Int_0^T (Y_s^n)^\top h(s,B_s,Y_s^n,Z_s^n)d\W_s\nonumber\\
 &- 2\Int_0^T(Y_s^n)^\top Z_s^n dB_s+\Int_0^T\|h(s,B_s,Y_s^n,Z_s^n)\|^2ds\nonumber\\
 &+\Int_t^T (Y_s^n)^\top g(s,B_s,Y_s^n,Z_s^n)\ast dB_s-2\Int_t^Ttrace[(Z^n_s)^\top g(s,B_s,Y^n_s,Z^n_s)]ds
 \end{split}
\end{align*}
Then, taking the square and the expectation yields
\begin{align*}
\begin{split}
\E\E^m\Big[\Big(n\Int_0^T &|Y_s^n-\pi(Y_s^n)|ds\Big)^2\Big]
\leq C \E\E^m\Big[|\Phi(B_{T})|^4\Big] + C\E\E^m\Big[\Big(\Int_0^T (Y_s^n)^\top f(s,B_s,Y_s^n,Z_s^n)ds\Big)^2\Big]\nonumber\\
&+ C\E\E^m\Big[\Big(\Int_0^T (Y_s^n)^\top h(s,B_s,Y_s^n,Z_s^n)d\W_s\Big)^2\Big]
+ C\E\E^m\Big[\Big(\Int_0^T(Y_s^n)^\top Z_s^n dB_s\Big)^2\Big]\nonumber\\
&+C\E\E^m\Big[\Big(\Int_0^T\|h(s,B_s,Y_s^n,Z_s^n)\|^2ds\Big)^2\Big]+C\E\E^m\Big[\Big(\Int_t^T (Y_s^n)^\top g(s,B_s,Y_s^n,Z_s^n)\ast d\B_s\Big)^2\Big]\nonumber\\
&+C\E\E^m\Big[\Big(\Int_t^Ttrace[(Z^n_s)^\top g(s,B_s,Y^n_s,Z^n_s)]ds\Big)^2\Big].
 \end{split}
\end{align*}
By using the isometry property and the boundedness of $f$, $g$ and $h$, we obtain
\begin{align*}
\begin{split}
\E\E^m\Big[\Big(n&\Int_0^T |Y_s^n-\pi(Y_s^n)|ds\Big)^2\Big]
\leq C \E\E^m\Big[|\Phi(B_{T})|^4\Big] + C\E\E^m\Big[\Big(\Int_0^T Y_s^n ds\Big)^2\Big]\nonumber\\
&+ C\E\E^m\Big[\Int_0^T |Y_s^n h(s,B_s,Y_s^n,Z_s^n)|^2ds\Big]+C\E\E^m\Big[\Int_0^T|Y_s^n Z_s^n|^2ds\Big]\nonumber\\
&+C\E\E^m\Big[\Int_0^T |Y_s^n g(s,B_s,Y_s^n,Z_s^n)|^2ds\Big]+C\E\E^m\Big[\Big(\Int_t^Ttrace[(Z^n_s)^\top g(s,B_s,Y^n_s,Z^n_s)]ds\Big)^2\Big]+C.
 \end{split}
\end{align*}
Finally, we deduce from Holder inequality and boundedness of $h$ that
\begin{align*}
\begin{split}
\E\E^m\Big[\Big(n\Int_0^T |Y_s^n-\pi(Y_s^n)|&ds\Big)^2\Big]
\leq C \E\E^m\Big[|\Phi(B_{T})|^4 +\Int_0^T |Y_s^n|^2 ds+\underset{0\leq t\leq t}{\Sup}|Y_t^n|^4+\Big(\Int_0^T\|Z_s^n  \|^2ds\Big)^2\Big]+C.
 \end{split}
\end{align*}
Thus, from  the estimate \eqref{estimateL4} we get the desired result.
\ep
\section{Useful results}
\label{useful results}
\noindent We denote by $J(B_{t,s}^{-1}(x))$ the determinant of the Jacobian
matrix of $B_{t,s}^{-1}(x)$, which is positive and in our particular context 
$J(B_{t,s}^{-1}(x))=1$. For $\varphi\in C_c^{\infty}(\R^d)$, we define
a process $\varphi_t:\,\Omega\times [t,T]\times \R^d\rightarrow \R$ by
\begin{equation}
\label{random:testfunction conv}
\varphi_t(s,x):=\varphi(B_{t,s}^{-1}(x))J(B_{t,s}^{-1}(x)).
\end{equation}

By a change of variable formula, we have for  all $v\in \mathbf{L}^2(\R^d)$,   $$(v\circ B_{t,s} (\cdot),
\varphi)=\Int_{\R^d}v(B_{t,s}
(x))\varphi(x)dx=\Int_{\R^d}v(y)\varphi(B_{t,s}^{-1}(y))J(B_{t,s}^{-1}(y))dy
=(v,\varphi_t(s,\cdot)), \; \, \P-a.s.$$ 
Since  $ (\varphi_t(s,x))_{ t\leq s}$ is a
process,  we may not use it directly as a test function because 
$\Int_t^T(u(s,\cdot),\partial_s\varphi_t(s,\cdot))ds$ has no sense. However
$\varphi_t(s,x)$ is a semimartingale and we have the following
decomposition of $\varphi_t(s,x)$ where the proof can be found  in \cite{BM} (proof of Lemma 2.1. p.135), see also Kunita \cite{K}, \cite{K1} for the use of such random test functions.
\begin{Lemma}\label{decomposition conv}
For every function $\varphi\in C_c^{\infty}(\R^d),$
\begin{equation}\label{decomp conv}\begin{array}{ll}
\varphi_t(s,x)&=\varphi(x)+\frac{1}{2}\displaystyle\Int_t^s\Delta\varphi_t(r,x)dr-\sum_{j=1}^{d}\Int_t^s\left(\sum_{i=1}^{d}\frac{\partial}{\partial
x_i}(\varphi_t(r,x))\right)dB_r^j,
 \end{array} \end{equation}

\end{Lemma}
\noindent Thanks to the above lemma, we can replace $\partial_s \varphi\, ds$ by the It\^o stochastic integral with respect to $d\varphi_t(s,x)$. This allows us to give the following
\begin{Definition}
For every $s\in[t,T]$, $u\in\Hc_T$ and $ \varphi\in C_c^{\infty}(\R^d),$ we define 
\begin{equation}\label{decomp conv}\begin{array}{ll}
\Int_s^T \left(u_r,d\varphi_t(r,.)\right)&=\displaystyle\Int_s^T \frac{1}{2}\left(u_r,\Delta\varphi_t(r,.)\right)dr-\sum_{j=1}^{d}\Int_s^T\left(\sum_{i=1}^{d}\left(u_r,\frac{\partial}{\partial x_i}(\varphi_t(r,.))\right)\right)dB_r^j,
 \end{array} \end{equation}
\end{Definition} 
\vspace{0.2cm}
\noindent We give now the following result which allows us to link in a natural way the solution of RSPDE with the associated RBDSDE.
 Roughly speaking,  for each  test function  $ \varphi \in C_c^{\infty}(\R^d) $,  the variational formulation \eqref{OPDE}  written with  the random test functions $\varphi_t(\cdot,\cdot)$  gives  the markovian RBDSDE \eqref{RBDSDE}  integrated against the test function $ \varphi $ , this dictionary can  be understood as the dual formulation of the variational equation for SPDE.  Pardoux and Peng  \cite{PP} have proved the probabilistic representation of classical solution $u$  for semilinear SPDEs  via BDSDEs by using the classical Itô's formula  for  $u (s, B_{t,s} (x))$.  However, since we consider Sobolev solutions for RSPDEs, the following proposition plays the rule of Itô's formula applied to the random test function $\varphi_t(s,x)$.  
    \begin{Proposition}
\label{weak:Itoformula1 conv} Let  Assumptions \ref{assxi}-\ref{assgener} hold and $(u,\nu)$ be a weak solution of the reflected 
SPDE (\ref{RSPDE1}) associated to $(\Phi,f,h)$, 
then for $s\in[t,T]$ and $\varphi\in
C_c^{\infty}(\R^d)$, 
\begin{align}\label{wspde}
\nonumber &\Int_{s}^{T}\!\!\Int_{\mathbb{R}^{d}}\!\big[\langle u(r,x), d\varphi _t(r,x)\rangle+\Frac{1}{2}\langle\nabla u(r,x), \nabla \varphi_t(r,x)\rangle\big]dxdr+\Int_{\mathbb{R}^{d}}\!\!\big[\langle u(s,x ), \varphi_t (s,x
)\rangle-\langle\Phi(x ), \varphi_t (T,x)\rangle\big]dx\\
\nonumber &=\Int_s^{T}\!\!\Int_{\mathbb{R}^{d}}\!\!\big[\langle f(r,x ,u(r,x),\nabla u(r,x)), \varphi_t(r,x)\rangle-\langle g(r,x ,u(r,x),\nabla u(r,x)), \nabla\varphi_t(r,x)\rangle\big]dxdr\\
&+\Int_{s}^{T}\!\!\Int_{\mathbb{R}^{d}}\!\!\langle h(r,x ,u(r,x),\nabla u(r,x)), \varphi_t(r,x)\rangle dx d\W_r
+\Int_{s}^{T}\!\!\Int_{\mathbb{R}^{d}}\!\langle\varphi_t (r,x)1_{\{u\in \partial D\}}(r,x), \nu(dr,dx)\rangle.
\end{align}
where  $\Int_{\R^d}\Int_s^Tu(r,x)d\varphi_t(r,x)dx $ is well defined in the semimartingale decomposition result (Lemma \ref{decomposition conv}).
\end{Proposition}
The proof of the proposition is  the same as the proof of Proposition 2.3, p. 137 in Bally and Matoussi  \cite{BM}.  This latter is based on Lemma 4.1 p.147 and Lemma 4.2. p.148  which involve the Wong-Zakai approximation of the Itô stochastic integral appearing in the semimartingale decomposition of the random test functions  given by \eqref{random:testfunction conv}  (see \cite{IW}, chap. 6, section 7 , p.480-517). The main idea is to use  $(\varphi_t(s,x)) $ as a test function in the  \eqref{OPDE}. The problem is that $(\varphi_t(s,x)) $  is not differentiable in the time variable $s$,  so that   $ \int \int_t^T u_s \partial_s \varphi_t  (s,x) ds dx $ has no sense. However     $(\varphi_t(s,x)) $ a semimartingale and one can use  Wong-Zakai approximation (see \cite{BM},  Lemma 4.2 p.148) to handel with this point .

\end{document}